\documentclass[a4paper,11pt]{article}
\usepackage[plainpages=false]{hyperref}
\usepackage{amsfonts,latexsym,rawfonts,amsmath,amssymb,amsthm,mathrsfs}
\usepackage{amsmath,amssymb,amsfonts,latexsym,lscape,rawfonts}

\usepackage[all]{xy}
\usepackage{eufrak}
\usepackage{makeidx}         
\usepackage{graphicx,psfrag}

\usepackage{array,tabularx}

\usepackage{setspace}

\newtheorem{thm}{Theorem}[section]
\newtheorem{cor}[thm]{Corollary}
\newtheorem{lem}[thm]{Lemma}

\newtheorem{prop}[thm]{Proposition}

\theoremstyle{remark}
\newtheorem{rmk}[thm]{Remark}

\theoremstyle{definition}
\newtheorem{Def}[thm]{Definition}                                        %

\def \C {\mathbb C}

\title{ On the regularity problem of   complex Monge-Ampere equations with conical singularities}
\author{Xiuxiong Chen,\ Yuanqi Wang}
\date{\textit{With admiration, we dedicate this article  to Prof Weiyue Ding, at  his 70th birthday.}}
\begin{document}

\maketitle{}
\begin{abstract}
In  the category of metrics with conical singularities along a smooth divisor with angle in $(0, 2\pi)$, we show that  locally defined  weak solutions ($C^{1,1}-$solutions)  to the  K\"ahler-Einstein equations actually possess maximum regularity, which means the metrics are actually H\"older continuous in the singular polar coordinates.  
This shows the weak K\"ahler-Einstein metrics constructed by Guenancia-Paun \cite{GP}, and independently by Yao \cite{GT}, are all actually strong-conical  K\"ahler-Einstein metrics. The key step is to establish a Liouville-type
theorem for weak-conical K\"ahler-Ricci flat metrics defined over $\C^{n}$, which depends on a Calderon-Zygmund  theory in the conical setting. 
\end{abstract}
\tableofcontents
\section{Introduction.}

 Consider the  singular space  $(\C\times \C^{n-1}, \omega_{\beta}) (\beta \in (0,1)),$
where $\omega_{\beta}$ is the standard flat background metric with conical singularities along $\{0\}\times \C^{n-1}$, written  as \[
\omega_{\beta} = {\beta^2\over {|z|^{2-2\beta}}} dz\otimes d\bar{z} +
 \Sigma_{j=1}^{n-1}d v_{j}\otimes d\bar{v}_{j}
,\]
where $z\in \C$ and $v_{j}$ are tangential variables to  $\{0\}\times \C^{n-1}$. Geometrically, this is a product of a flat two-dimension cone with Euclidean $\mathbb{C}^{n-1}$. From now on, we denote the singular divisor  $\{0\}\times \C^{n-1}$ as $D$. In this introduction, we take the balls to be centered at the origin, with respect to $\omega_{\beta}$. For more detailed notations, please refer to section 2 of this article.

We want to understand the PDE theory in this space, using intrinsic metric.  For any domain $\Omega \in \C\times \C^{n-1},$  the complex Monge Ampere equation take a simpler form
\begin{equation}\label{equ conical Monge Ampere}
  \det (\phi_{i \bar j}) = {f \over |z|^{2-2\beta}},
\end{equation}
where \begin{equation}\label{equ every Kahler metric has a potential}\omega_\phi = \sqrt{-1}\partial \bar \partial \phi\end{equation} gives a K\"ahler metric in $\Omega$ with conical angle $2\pi\beta$ along $D$. The Laplacian operator
of $\omega_{\beta}$ is 
\[
\triangle_\beta = \frac{|z|^{2-2\beta}}{\beta^2} {{\partial^2}\over {\partial z \partial \bar z}} +   \Sigma_{j=1}^{n-1} {{\partial^2}\over {\partial v_{j} \partial {\bar v}_{j}}}.
\]

Sometimes we also use the real laplacian of $\omega_{\beta}$, denoted as 
$\Delta$. Notice that $\Delta=4\Delta_{\beta}$. 
\begin{Def}\label{def local formula for conical KE metrics} For any constant $\lambda > 0$, suppose $\phi$ solves (\ref{equ conical Monge Ampere}) with $$f  = e^{\lambda \phi+h}, \ h\in C^{\infty}(\Omega)\ \textrm{and}\
\sqrt{-1}\partial \bar \partial h=0,
$$  then $\omega_\phi$ is a conical K\"ahler-Einstein metric with scalar curvature $ -n \lambda.\;$
When $\lambda=0$,  $\omega_\phi$ is a conical K\"ahler-Ricci flat metric.  
\end{Def}
\begin{rmk}Notice that the conical K\"ahler-Einstein metrics (along  smooth divisors) considered 
in all the references we know (including \cite{Don}, \cite{Brendle}, \cite{CGP}, \cite{GP}, \cite{CDS2}, \cite{JMR}, \cite{LiSun},\cite{LiuZhang}, \cite{SongWang}, \cite{Yao}, \cite{Berman}...), can be written as in Definition \ref{def local formula for conical KE metrics} near the $D$, under holomorphic coordinates.
\end{rmk}

To state our main results, we define the following. 
\begin{Def}\label{def weak conical metrics}(Weak-conical K\"ahler metrics) A function $u$ defined in $\Omega $ is called a $C^{1,1,\beta}(\Omega)$-function if it satisfies 
\begin{itemize}
\item  $u\in C^{2,\alpha}(\Omega\setminus D)\cap  C^{\alpha}(\Omega)$, for some $1>\alpha>0$;
\item  $-K\omega_{\beta}\leq \sqrt{-1}\partial\bar{\partial}u \leq K\omega_{\beta}$ over $\Omega\setminus D$. The minimum of all such $K$ is defined as our $C^{1,1,\beta}(\Omega)$-seminorm and denoted as $[\ \cdot\ ]_{C^{1,1,\beta}(\Omega)}$.
\end{itemize}

A closed positive $(1,1)$-currrent  $\omega$ defined in $\Omega$ is called a weak-conical K\"ahler metric if $\omega$ admits a plurisubharmonic $C^{1,1,\beta}-$potential (in the sense of (\ref{equ every Kahler metric has a potential})) near any $p\in \Omega$, and 
\[\frac{\omega_{\beta}}{K}\leq \omega\leq K\omega_{\beta}\ \textrm{over}\ \Omega\setminus D,\ \textrm{for some}\ K\geq 1.\] Sometimes we call such metrics as $L^{\infty,\beta}$-metrics, with norm defined as the $C^{1,1,\beta}(\Omega)$-seminorm  in the previous paragraph with respect to the potentials.
\end{Def}
\begin{rmk}Notice that for a function, being $C^{1,1,\beta}$ is stronger (away from $D$) than being  $C^{1,1}$ in the usual sense, even in the smooth case (when $\beta=1$). Namely, we require the function 
to be $C^{2,\alpha}$ away from the singularity. The $C^{1,1,\beta}$ and $L^{\infty,\beta}$ spaces are really adapted to the conic case only.

The above definition is the same as in \cite{WYQWF} and \cite{ChenWanglongtime}, we just formulate it here to include the definition of $C^{1,1,\beta}$ functions. 
\end{rmk}

\begin{Def}\label{equ CKS operator}(CKS operators) Similar to definition \ref{def weak conical metrics}, we say $L$ is a Conelike K\"ahler Second-order  operator over a ball $B$, if $L=A^{i\bar{j}}\frac{\partial^2 }{\partial z_{i}\partial \bar{z}_{j}}$ such that 
\begin{enumerate}
\item $A^{i\bar{j}}\in C^{\alpha}(B\setminus D)$ is a Hermitian matrix valued function. 
\item $-K\omega_{\beta}^{i\bar{j}}\leq  A^{i\bar{j}}\leq K\omega_{\beta}^{i\bar{j}}$ as Hermitian matrix functions over $B$, for some constant $K\geq 1$.
\end{enumerate}
We define the $L^{\infty,\beta}_{O}$-norm of a CKS operator $L$ as the infimum of the constant $K$ in the item $2$ above. The laplace operator of any weak conical-K\"ahler metric is an elliptic CKS operator, but in general a CKS operator does not have to be elliptic. 
\end{Def}
According to \cite{CDS2}, if a conical K\"ahler-Einstein metric is in  $C^{\alpha, \beta}\;$ for some $\alpha > 0$, then it is necessarily in $C^{\alpha', \beta}$ for
all $ \alpha' \in (0, \min(1, {1\over \beta}-1)).\;$ The fundamental problem is when $\alpha = 0, \;$ in other words, when  the metric tensor is barely $L^{\infty,\beta}$, does the metric actually possess higher regularity?
This is of course a core problem in the study of conical K\"ahler geometry.  In this paper, we prove

\begin{thm}\label{thm general regularity of conical Monge-Ampere equations}Let $\Omega$ be an open set in $\C^{n}$. Suppose $f \in C^{1,1,\beta}(\Omega)$, $f>0$.   For any  solution  $\phi\in C^{1,1,\beta}(\Omega)$ to equation (\ref{equ conical Monge Ampere}) such that $\sqrt{-1}\partial\bar{\partial}\phi$ is a weak conical metric, $\phi$ is actually in $C^{2,\alpha,\beta}(\Omega)$, for all $\alpha$ such that  $ 0< \alpha < \min({1\over \beta}-1, 1).\;$ 
\end{thm}
\begin{rmk}Theorem \ref{thm general regularity of conical Monge-Ampere equations} does not give any $C^{2,\alpha,\beta}-$norm bound on $\phi$, it only says $\phi$ has $C^{2,\alpha,\beta}-$regularity in the open set $\Omega$. Actually, the norm bounds and the apriori estimate are already proved in \cite{ChenWanglongtime}, from page 13 to page 19. The point of Theorem \ref{thm general regularity of conical Monge-Ampere equations} is the regularity, but not the norm bounds.

\end{rmk}

 Theorem \ref{thm general regularity of conical Monge-Ampere equations} has an immediate corollary. For the sake of accuracy, we prefer to state it in a more geometric way.

\begin{cor}\label{Cor regularity of weak KE metrics}
Any weak-conical K\"ahler-Einstein metric in a domain $\Omega \subset \C\times \C^{n-1}$ must be a $C^{\alpha,\beta}$ conical K\"ahler-Einstein metric, for any $ 0< \alpha < \min({1\over \beta}-1, 1).\;$
\end{cor}
\begin{rmk}Using Yau-type Schwartz lemmas and some tricky oberservations, Guenancia-Paun   constructed weak-conical K\"ahler-Einstein metrics in \cite{GP}. Yao also independently constructed weak-conical K\"ahler-Einstein metrics in \cite{Yao}, using interesting tricks.  Corollary \ref{Cor regularity of weak KE metrics} implies when the divisor $D$ is smooth,  both  Guenancia-Paun and Yao's 
weak-conical K\"ahler-Einstein metrics are (strong) conical metrics i.e 
they are all H\"older continous  metrics.
\end{rmk}

In Theorem \ref{thm general regularity of conical Monge-Ampere equations} and Corollary \ref{Cor regularity of weak KE metrics},  we only assume the underlying metric tensor is  $L^{\infty,\beta}$.  A crucial step is to prove  the following Liouville type theorem:

\begin{thm}\label{thm Liouville}(Strong Liouville Theorem)  Suppose $\omega$ is a weak-conical K\"ahler metric over $\C^{n}$, and $\omega$ satisfies

\begin{equation}\label{equ KRF equation}
 \omega^n = \omega_{\beta}^n,\  \frac{\omega_{\beta}}{K}\leq \omega \leq K\omega_{\beta}\
 \textrm{over} \
\mathbb{C}^n \setminus D,\end{equation}
for some $K\geq 1$. Then, there is a linear transformation $L$ which preserves $\{z=0\}$ and $\omega=L^{\star}\omega_{\beta}$.
\end{thm}

\begin{rmk}This strong Liouville Theorem is first proved by Chen-Donaldson\\ -Sun in \cite{CDS2}, with the additional assumption that 
$\omega$ is a metric cone. Later, assuming $\omega$ has $C^{\alpha,\beta}$-regularity for some $\alpha>0$ instead of being a metric cone, the Liouville Theorem is proved by the authors  in Theorem 1.14 in  \cite{ChenWanglongtime}.
\end{rmk}
This strong Liouville  theorem is much harder, since  we assume the underlying metric tensor is only $L^{\infty,\beta}$.  In particular, we can not take any more derivatives to the  Einstein equation (\ref{equ KRF equation}) globally, so existing methods are not sufficient  anymore. For this purpose, we need to develop $W^{2,p}-$estimate in the conical settings.  In \cite{Don}, Donaldson developed the Schauder theory for conical Laplace operator,  and used that to deform the cone angle of 
conical K\"ahler-Einstein metrics.  In this paper, we  establish the  corresponding conical $W^{2,p}$-theory.  The definition of $W^{k,p,\beta} (k=1,2)$ is given in Section 2. 
    To prove the $W^{2,p,\beta}-$estimate, it sufficies to consider the following set of second order operators of non-purely normal $(1,1)-$derivatives as in \cite{Don}.
\begin{eqnarray}\label{Def of mathbb T}& &\mathbb{T}\nonumber
\\&=&\{\frac{\partial^2}{\partial w_i\partial r},\ 1\leq i\leq 2n-2;\ \frac{\partial^2}{\partial w_i \partial w_j},\ 1\leq i,j\leq 2n-2;\ \frac{1}{r}\frac{\partial^2}{\partial w_i\partial \theta},\nonumber
\\& & 1\leq i\leq 2n-2\},\end{eqnarray}
where $r=|z|^{\beta}$, and the $\theta$ is the angle of $z$. There will be more detailed definition in section 2. 

   Following Chap 9 of \cite{GT}, we define a class of operators $T$ as
\begin{equation}\label{equ definition of the operator T}Tf=\mathfrak{D} N_{\beta,B}f,\ \mathfrak{D}\in \mathbb{T},
\end{equation}
where $N_{\beta,B}f$ is the Newtonian potential of $f$, defined by convolution with the Green's function as in Definition \ref{def Newtonian potential}. 

   Actually, the operator $T$ and its dual  $T^{\star}$ are both very similar to the singular integral operators considered by Calderon-Zygmund in \cite{CalderonZygmund}, and by Stein \cite{Stein}(see Theorem $1$ in section 2.2 of \cite{Stein}). Though  our conical case is different from the classical cases on several aspects, the really surprising thing is: the  proof of Theorem 9.9 in \cite{GT} proceeds well in our case, after overcoming several analytical 
difficulties. Namely, the following $W^{2,p,\beta}$-estimate is true.
\begin{thm}\label{thm W2p estimate of general L with small oscillation}
Suppose $L$ is an elliptic CKS operator defined over $B(2)$. Suppose there is a sufficiently small constant $\delta_{0}$ such that $$|L-\Delta_{\beta}|_{L^{\infty,\beta}_{O}}\leq \delta_{0},\ \textrm{over}\ B(2).$$  
Suppose  $u\in C^{2}(B(2)\setminus D)\cap W^{2,p,\beta}[B(2)]$ is  a classical-solution to \[Lu=f\ \textrm{in}\ B(2)\setminus D,\ f\in L^{p}[B(2)],\   \infty>p\geq 2.\] Then 
\begin{equation*}
[u]_{W^{2,p,\beta},B(1)}\leq C (|f|_{L^{p},B(2)}+|u|_{W^{1,p},B(2)}),
\end{equation*}
where $C$ only depends on $n,\beta,p$. In particular, we have 
\begin{equation*}
[u]_{W^{2,p,\beta},B(1)}\leq C (|\Delta_{\beta}u|_{L^{p},B(2)}+|u|_{W^{1,p},B(2)}).
\end{equation*}
\end{thm}

 The following Sobolev-imbedding Theorem in the conical category is also crucially needed in the proof of Theorem \ref{thm general regularity of conical Monge-Ampere equations}. 
 \begin{thm}\label{thm C1alpha estimate with Lp right  hand side}Let $u\in W^{1,2}(B(2))\cap C^{2}(B(2)\setminus D)$ be a weak solution to $\Delta_{\beta}u=f$ in $B(2)$, $f\in L^{p},\ p>2n$. Then for all $\alpha<\min\{1-\frac{2n}{p}, \frac{1}{\beta}-1\}$, we have $u\in C^{1,\alpha,\beta}(B(1))$ and 
\[|u|_{1,\alpha,\beta,B(1)}\leq C(|u|_{W^{1,2}, B(2)}+|f|_{L^p, B(2)}).\]
\end{thm}

\textbf{Convention of the constants:} The $"C"$'s in the estimates mean constants independent of the object estimated, suppose the object satisfies the conditions and bounds  in the correponding statement. In some cases we  say explicitly what does the $"C"$ depend on. When we don't say anything to the $"C"$, we mean it can depend on the conditions and bounds  in the corresponding statement, for example, like the  $C^{1,1,\beta}$-bound on the given potential, or the $C^{\alpha}-$bound on the given metric given  away from the divisor, or the $C^{1,1,\beta}$-bound on the given volume form $f$, or the quasi-isometric constant of $\omega$ with respect to $\omega_{\beta}$,... and so on.  

\textbf{Distances and Balls:} In most of the cases, we use distance
and balls defined by the model cone metric $\omega_{\beta}$, unless otherwise specified. The balls are usually centered at  the origin, unless otherwise specified. The only big exception is in section \ref{section Calderon Zygmund and potential}, where we use the  Euclidean metric $\omega_{E}$ in the polar coordinates. The reason is that it's super convenient to consider cube decomposition with respect to the  Euclidean metric $\omega_{E}$ in the polar coordinates, which is necessary in the Calderon-Zygmund theory.    $\omega_{E}$ and $\omega_{\beta}$ are  quasi-isometric to each other i.e 
\[\beta\omega_{\beta}\leq \omega_{E}\leq \frac{\omega_{\beta}}{\beta}.\]
Thus the distances defined by them are actually equivalent. 


\section{The $L^{2}$-estimate.}
In this section, we fix the necessary notations and prove the $L^2$-estimate of the conical Laplace equation in Lemma \ref{lem W22 estimate of the newtonian potential}. This is the first step toward a full $W^{2,p,\beta}-$theory for all $p\in (1,\infty)$.
 
Let $r=|z|^{\beta}$ and $\theta$ be just  the angle of $z$ from the positive real axis. In the polar coordinates $r,\theta, w_{i},\ 0\leq i\leq 2n-2$, $\omega_{\beta}$ can be written as 
$$\omega_{\beta}=dr^2+\beta^2r^2d\theta^2+\Sigma_{i=1}^{2n-2}dw_{i}\otimes dw_{i},$$
where $r$ is the distance to the divisor $D=\{0\}\times \C^{n-1}$, $\theta$ is the usual angle of the variable $z$, and $w_{i}$ are the tangential variables. 

 Notice in the polar coordinates we have $\beta^2g_{E}\leq \omega_{\beta}\leq \frac{1}{\beta^2}g_{E}$, where $g_{Euc}$ is Euclidean metric in the polar coordinates i.e
   \begin{equation}\label{equ def of g_E in polar coordinates}g_{E}=dr^2+r^2d\theta^2++\Sigma_{i=1}^{2n-2}dw_{i}\otimes dw_{i}.\end{equation} 
   We denote $\omega_{E}$ as the K\"ahler form of $g_{E}$.
   We will be frequently using the polar coordinates in most of the following content, as in this nice coordinates,  the conical metrics are  quasi-isometric to the Euclidean metric $g_{Euc}$. 
We first define the space $W^{1,p,\beta}(B)$ as usual $W^{1,p}$-space 
in the polar coordinates, $\infty>p\geq 2$. 
 \begin{Def}\label{W2pbeta} ($W^{2,p,\beta}-$space). Given $p\geq 2$, and a ball $B$, a function $u$ is said to be in the space $W^{2,p,\beta}(B)$ if the following holds. We can understand  the polar coordinates as the intrinsic coordinates of $\omega_{\beta}$. 
\begin{itemize}
\item For any $\epsilon$, $u\in W^{2,2}[B\setminus T_{\epsilon}(D)]$, where $T_{\epsilon}(D)$ is the $\epsilon-$tubular neighborhood of $D$.
 
\item $|z|^{2-2\beta}\frac{\partial^2 u}{\partial z \partial\bar{z}}\in L^{p}(B)$;
\item $|z|^{1-\beta}\frac{\partial^2 u}{\partial z \partial w_j}\in L^{p}(B)$, for all $0\leq j\leq 2n-2$;
\item $\frac{\partial^2 u}{\partial w_{i} \partial w_{j}}\in L^{p}(B)$, for all $0\leq i,j\leq 2n-2$;
\item $u\in W^{1,p,\beta}(B)$. 
\end{itemize}
The semi norm is expressed as  
\begin{eqnarray}\label{equ def of seminorm}
& &[u]_{W^{2,p,\beta}}
\\&=&||z|^{2-2\beta}\frac{\partial^2 u}{\partial z \partial\bar{z}}|_{L^{p}(B)}+\Sigma_{j=1}^{2n-2}||z|^{1-\beta}\frac{\partial^2 u}{\partial z \partial w_j}|_{L^{p}(B)}+\Sigma_{i,j=1}^{2n-2}|\frac{\partial^2 u}{\partial w_{i} \partial w_{j}}|_{L^{p}(B)}.
\nonumber
\end{eqnarray}
\end{Def}

\begin{lem}\label{lem completeness of W2,p,beta} For any ball $B$, $W^{2,p,B}(B)$ is a (complete) Banach space.
\end{lem}
\begin{rmk}This  completeness lemma is used in the definition of $N_{\beta}f$ for $f\in L^{p}$, $p\geq 2$,  as in Lemma \ref{lem W22 estimate of the newtonian potential}. We  present a full proof for the convenience of the readers, though it's straightforward. 

\end{rmk}
\begin{proof}{of Lemma \ref{lem completeness of W2,p,beta}:}
Without loss of generality, we assume $B$ is the unit ball (centered at the origin). We only  consider the case $W^{2,2,\beta}(B)$, the proof for all $p$ is exactly the same.  It suffices to construct a limit. Suppose $\{u_{k}\}$ is a Cauchy-Sequence of $W^{2,2,\beta}(B)$, then in the sense of 
$W^{1,2,\beta}(B)$, $u_{k}$ admits a limit denoted as $\underline{u}$. Then it remains to show $\underline{u}$ is actually in $W^{2,2,\beta}(B)$.

 Denote $B_{R}$ and the radius of $R$, and  $T_{R}(D)$ as the turbular neighborhood of $D$ with radius $R$ (as in Definition \ref{W2pbeta}). Over $B_{1-\frac{\epsilon}{2}}\setminus T_{\frac{\epsilon}{2}}(D)$, we deduce that  $\{\Delta_{\beta}u_{k}\}$ is a Cauchy-Sequence in $L^{2}[B_{1-\frac{\epsilon}{2}}\setminus T_{\frac{\epsilon}{2}}(D)]$.  Then we apply the interior elliptic estimate to the pair of domains $$B_{1-\frac{\epsilon}{2}}\setminus T_{\frac{\epsilon}{2}}(D),\ B_{1-\epsilon}\setminus T_{\epsilon}(D).$$ 

 By Theorem 8.8 in \cite{GT}, we deduce
\begin{eqnarray*}
& &|u_{k}-u_{l}|_{2,2,B_{1-\epsilon}\setminus T_{\epsilon}(D))}
\\&\leq & C(\epsilon)[|u_{k}-u_{l}|_{1,2,B_{1-\frac{\epsilon}{2}}\setminus T_{\frac{\epsilon}{2}}(D)}+|\Delta_{\beta}u_{k}-\Delta_{\beta}u_{l}|_{0,2,B_{1-\frac{\epsilon}{2}}\setminus T_{\frac{\epsilon}{2}}(D)}].
\end{eqnarray*} 
Thus, $\{u_{k}\}$ is a Cauchy-Sequence in the usual Sobolev space
$W^{2,2}(B_{1-\epsilon}\setminus T_{\epsilon}(D))$. Then, by the completeness of the
usual Sobolev spaces, and the diagonal sequence trick,  there exists a limit function in $W^{2,2}_{loc}(B\setminus D)$, which can be nothing else than $\underline{u}$,   with the following property. 
\[\lim_{k\rightarrow \infty}|u_{k}-\underline{u}|_{2,2,B_{1-\epsilon}\setminus T_{\epsilon}(D)}=0,\ \textrm{for any}\ \epsilon>0.\]
Since over $B_{1-\epsilon}\setminus T_{\epsilon}(D)$, the $W^{2,2,\beta}(B_{1-\epsilon}\setminus T_{\epsilon}(D))$-norm is weaker than the usual $W^{2,2}(B_{1-\epsilon}\setminus T_{\epsilon}(D))-$norm, and $\{u_{k}\}$ is a Cauchy-Sequence in $W^{2,2,\beta}(B)$, we deduce the following by the Minkovski inequality
\begin{eqnarray}\label{eqn estimate of uinfty W22beta norm independent of chop off}& &\nonumber
|\underline{u}|_{W^{2,2,\beta}(B_{1-\epsilon}\setminus T_{\epsilon}(D))}
\\&\leq &\limsup_{k} \{|\underline{u}-u_{k}|_{W^{2,2,\beta}(B_{1-\epsilon}\setminus T_{\epsilon}(D))}\nonumber
+|u_{k}|_{W^{2,2,\beta}(B_{1-\epsilon}\setminus T_{\epsilon}(D))}\}
\\&=&  \limsup_{k} 
|u_{k}|_{W^{2,2,\beta}(B_{1-\epsilon}\setminus T_{\epsilon}(D))}\leq C,
\end{eqnarray}
where $C$ does not depend on $\epsilon$. Since $\epsilon$ is arbitrary, (\ref{eqn estimate of uinfty W22beta norm independent of chop off}) implies 
\begin{equation}
|\underline{u}|_{W^{2,2,\beta}(B)}\leq C<\infty, \ \textrm{then}\ \underline{u}\in W^{2,2,\beta}(B).  
\end{equation}

The proof of Lemma \ref{lem completeness of W2,p,beta} is complete.
\end{proof}
To study the $W^{2,p,\beta}$-estimate, we quote the heat kernel formula 
in \cite{ChenWang}. Denote $x=(r,\theta,\underline{x})$ and $y=(r^{\prime},\theta^{\prime},\underline{x}^{\prime})$, where $\underline{x}$ is the tangential projection of $x$. Denote $R=|\underline{x}-\underline{x}^{\prime}|$. The heat kernel is 
\begin{eqnarray}\label{eqn heat kernel formula} & & H(x,y,t)\nonumber
\\&=&\frac{1}{(4\pi t)^{n}}e^{-\frac{r^2+{{\acute{r}}}^2+R^2}{4t}}
\{\Sigma_{k,\ -\pi <\beta[\theta-\theta^{\prime}]+2k\beta \pi<\pi}e^{\frac{rr^{\prime}\cos(\beta[\theta-\theta^{\prime}]+2k\beta \pi)}{2t}}
\nonumber
\\&+& K(\frac{rr^{\prime}}{2t},\beta[\theta-\theta^{\prime}])
+\frac{1}{2} \Sigma_{k,\beta[\theta-\theta^{\prime}]+2k\beta \pi=\pm\pi}e^{-\frac{rr^{\prime}}{2t}}\},
\end{eqnarray}

where
\begin{eqnarray}& &K(\frac{rr^{\prime}}{2t},\beta[\theta-\theta^{\prime}])
\nonumber
\\&=&
\frac{\sin\frac{\pi}{\beta}}{\pi\beta}\int_{0}^{\infty}e^{-\frac{rr^{\prime}}{2t}\cosh y}\frac{[\cos\frac{\pi}{\beta}-\cos[\theta-\theta^{\prime}]\cosh\frac{y}{\beta}]}{[\cosh\frac{y}{\beta}-\cos\frac{\beta[\theta-\theta^{\prime}]-\pi}{\beta}][\cosh\frac{y}{\beta}-\cos\frac{\beta[\theta-\theta^{\prime}]+\pi}{\beta}]}\nonumber
\\& &\cdot dy.
\end{eqnarray}

In the above formula, we actually abused the notation a little bit, as in \cite{ChenWang}. To be precise,  the $"\theta-\theta^{\prime}"$ means the unique angle in $(-\pi,\pi]$ which is mod $2\pi$ equivalent to 
$\theta-\theta^{\prime}$.

We define the Green function of $\omega_{\beta}$ as
\[\Gamma(x,y)=-\int^{\infty}_{0}H(x,y,t)dt.\]

The following lemma is true. 
\begin{lem}\label{lem boundary int of Green func tends to 1}    For any $x\notin D$, we have 
\[\lim_{\epsilon\rightarrow 0} \int_{\partial B_{x}(\epsilon)}\frac{\partial \Gamma(x,y)}{\partial \nu_{y}}dy=1. \]
\end{lem}
\begin{proof}{of Lemma \ref{lem boundary int of Green func tends to 1}:}
   It sufficies to notice that, by the assumption that $x\notin D$, 
we have $r_{x}>0$ ($r_{x}=r$, we just add the sub $x$ to emphasize its dependence on $x$). Then, when $k\neq 0$, we deduce
\begin{equation}
e^{-\frac{r^{2}+r^{\prime 2}}{4t}}e^{\frac{rr^{\prime}\cos(\beta[\theta-\theta^{\prime}]+2k\beta \pi)}{2t}}\leq e^{-\frac{(r-r^{\prime })^2+2(1-\cos\beta\pi)rr^{\prime}}{4t}}\leq e^{-\frac{a}{t}}, 
\end{equation}   
where  $a$ is a  positive constant depending on $x$, especially $r_{x}$ (the distance from $x$ to $D$). Then, by defining 
\begin{eqnarray}& &\Gamma_{E}
\\&=& -\int_{0}^{\infty}\frac{1}{(4\pi t)^{n}}e^{-\frac{r^2+{{\acute{r}}}^2+R^2}{4t}}
\{\Sigma_{k\neq 0,\ -\pi <\beta[\theta-\theta^{\prime}]+2k\beta \pi<\pi}e^{\frac{rr^{\prime}\cos(\beta[\theta-\theta^{\prime}]+2k\beta \pi)}{2t}}
\nonumber
\\&+& K(\frac{rr^{\prime}}{2t},\beta[\theta-\theta^{\prime}])
+\frac{1}{2} \Sigma_{k,\beta[\theta-\theta^{\prime}]+2k\beta \pi=\pm\pi}e^{-\frac{rr^{\prime}}{2t}}\}dt,
\end{eqnarray}
we obtain  when $y\in B_{x}(\frac{r_{x}\sin\beta\pi }{2})$ and $\beta[\theta-\theta^{\prime}]\neq \pm \pi\ mod\ 2\beta\pi$ that 
\begin{eqnarray}& &|\nabla_{y}\Gamma_{E}|\nonumber
\\& \leq &\int_{0}^{\infty}|\nabla_{y}\{\frac{1}{(4\pi t)^{n}}e^{-\frac{r^2+{{\acute{r}}}^2+R^2}{4t}}
[\Sigma_{k\neq 0,\ -\pi <\beta[\theta-\theta^{\prime}]+2k\beta \pi<\pi}e^{\frac{rr^{\prime}\cos(\beta[\theta-\theta^{\prime}]+2k\beta \pi)}{2t}}
\nonumber
\\&+& K(\frac{rr^{\prime}}{2t},\beta[\theta-\theta^{\prime}])
\}|dt\nonumber
\\&\leq & C\int_{0}^{\infty}(\frac{1}{t^{n}}+\frac{1}{t^{n+\frac{1}{2}}}+\frac{1}{t^{n+1}})e^{-\frac{a}{t}}dt\leq C_{a}.
\end{eqnarray}
By continuity, we deduce for all $x\notin D$ and $y\in B_{x}(\frac{r_{x}\sin\beta\pi }{2})$ that 
\begin{equation}\label{equ bound on Gamma E}
|\nabla_{y}\Gamma_{E}(x,y)|\leq C_{a}.
\end{equation} 
Notice that  
\begin{eqnarray}& &\Gamma(x,y) \nonumber
\\&=& -\int^{\infty}_{0}\frac{1}{(4\pi t)^{n}}e^{-\frac{r^2+{{\acute{r}}}^2+R^2}{4t}}
e^{\frac{rr^{\prime}\cos(\beta[\theta-\theta^{\prime}])}{2t}}dt+\Gamma_{E}.\nonumber
\\&=&-\int^{\infty}_{0}\frac{1}{(4\pi t)^{n}}e^{-\frac{|x-y|^2}{4t}}
dt+\Gamma_{E}.\nonumber
\\&=& -\frac{1}{4\pi^{n}\rho^{2n-2}}\Gamma(n-1)+\Gamma_{E},
\end{eqnarray}
where $\rho=|x-y|$ and $\Gamma(n)$ is the Gamma-function. Using (\ref{equ bound on Gamma E}), we deduce $$\lim_{\epsilon\rightarrow 0} \int_{\partial B_{x}(\epsilon)}\frac{\partial \Gamma_{E}(x,y)}{\partial \rho}dy=0.$$
Moreover, we have  $\frac{\Gamma(n)S(2n-1)}{2\pi^{n}}=1$, where $S(2n-1)$ is the area of $(2n-1)$-dimensional unit sphere. Then we compute 
\begin{eqnarray}
& &\lim_{\epsilon\rightarrow 0} \int_{\partial B_{x}(\epsilon)}\frac{\partial \Gamma(x,y)}{\partial \rho}dy\nonumber
\\&=& \lim_{\epsilon\rightarrow 0} \int_{\partial B_{x}(\epsilon)}\frac{\partial \Gamma_{E}(x,y)}{\partial \rho}dy+\frac{1}{2\pi^{n}\rho^{2n-1}}\Gamma(n)S(2n-1)\rho^{2n-1}\nonumber
\\&=& 1.
\end{eqnarray}

\end{proof}
\begin{Def}\label{def Newtonian potential}  We denote $N_{\beta,\Omega}f$ as the Newtonian potential of $f$ over $\Omega$ i.e
\[N_{\beta,\Omega}f=\int_{\Omega}\Gamma(x,y)f(y)dy.\]
\end{Def}
\begin{lem}\label{lem Green's representation}(Green Representation) Suppose $u\in W^{2,2,\beta}_{c}(C^{n})\cap C^{2}(\C^{n}\setminus D)$, then the Green's representation formula 
holds for $u$ i.e for all $x\notin D$, we have
\begin{equation}\label{equ Green's formula}  u(x)=N_{\beta,\C^{n}}(\Delta_{\beta}u)(x).
\end{equation}

\end{lem}
\begin{proof}{of Lemma \ref{lem Green's representation}:}  
First, since $x\notin D$ and $u\in C^{2}(C^{n}\setminus D)$, then when  $\epsilon_{0}$ such that  $B_{x}(\epsilon_{0})\cap D=\emptyset$, the following 
  \[N_{\beta,\C^{n}}\Delta_{\beta}u=\int_{B_{x}(\epsilon_{0})}\Gamma(x,y)\Delta_{\beta}u(y)dy+
  \int_{\C^{n}\setminus B_{x}(\epsilon_{0})}\Gamma(x,y)\Delta_{\beta}u(y)dy\]
   is well-defined pointwisely for all $x\in \C^{n}\setminus D$. 
   
Since integration by parts holds for $u\in W^{2,2,\beta}_{c}(C^{n})\cap C^{2}(C^{n}\setminus D)$, then (\ref{equ Green's formula}) follows from the well known derivation of formula (2.17) in page 18 of \cite{GT}, and Lemma \ref{lem boundary int of Green func tends to 1}.

\end{proof}

In the conical case, the operator $T$ (as defined in \ref{equ definition of the operator T}) might not be self adjoint because there is one
special direction. Nevertheless, this could compensated by the good 
properties of $T^{\star}$. For any $f,g\in C_{c}^{\alpha,\beta}(B)$, we have 
\begin{equation}
\int_{B}(Tf)g dx=\int_{B}f T^{\star}g dy.   
\end{equation}
It's easy to show that 
\[T^{\star}g=-D_{y_j}\int_{B}D_{x}\Gamma(x,y)g(x)dx,\]
where $y_{j}$ is a tangential varible in the $y$-component, and $D_{x}$
is an order $1$ differential operator in the $x$-component.

Notice $D_{x}$ can not be integrated by parts in general, since $div \frac{\partial}{\partial r}\neq 0$ and $div\{\frac{1}{r}\frac{\partial}{\partial \theta}\}\neq 0$. Nevertheless,  Lemma \ref{lem Tstar maps calpha to itself}  guarantees that $T^{\star}$ is densely defined in $L^{2}(B)$, which leads to our necessary $L^{2}-$estimate. The proof of the following crucial $L^{2}-$estimate is almost the same as proof i of Theorem 9.9 in \cite{GT}. Nevertheless, since it concerns the correct choice of the Hessian operator to integrate by parts toward, we still present a detailed proof. The  Hessian operator we choose here leads to 
the necessary  $W^{2,p,\beta}$-estimate when $p$ is large. 

\begin{lem}\label{lem W22 estimate of the newtonian potential} Given a ball $B$, suppose $f\in L^{2}(B)$, then $N_{\beta,B}f$ is well-defined. Moreover, 
\[|Tf|^{2}_{L^{2}(B)}\leq C\int_{B}f^2,\]
and 
\[|T^{\star}f|^{2}_{L^{2}(B)}\leq C\int_{B}f^2.\]
\end{lem}
\begin{proof}{of Lemma  \ref{lem W22 estimate of the newtonian potential}:}
For any sequence $\epsilon_{k}>0$ such that $\epsilon_{k}\rightarrow 0$, we consider the smoothing of cutoffs of $f$ with parameter $\epsilon_{k}$, denoted as $f_{\epsilon_{k}}$.
The point is that the smoothing and cutoffs work well in the conical case. Namely, the approximation functions $f_{\epsilon_{k}}$ are in $C^{\infty}_{c}(B)$, and $$\lim_{k\rightarrow \infty}|f_{\epsilon_{k}}-f|_{L^{2}(B)}=0.$$
The space $C^{\infty}(B)$ is of compact supported smooth  functions in the polar coordinates, not holomorphic coordinates. 

Step 1: Then we consider $\omega_{\epsilon_{k}}=N_{\beta,B}f_{\epsilon_{k}}$. Then, by the work in Donaldson (also see \cite{ChenWang}), $N_{\beta,B}f_{\epsilon_{k}}\in C^{2,\alpha,\beta}(B)$, thus it makes sense to consider Hessian of $\omega_{\epsilon_{k}}$ in some sense. 
It sufficies to prove 
\begin{equation}\label{equ W22 seminorm identity}\int_{B}f_{\epsilon_{k}}^2\omega_{\beta}^n
=\int_{\C^{n}}|\Delta w_{\epsilon_{k}}|^2\omega_{\beta}^n
=\int_{\C^{n}}|\nabla^{1,1,\beta} w_{\epsilon_{k}}|^2\omega_{\beta}^n,
\end{equation}
where the $\nabla^{1,1,\beta}$ is the Hessian operator whose components are exactly those  in the seminorm (\ref{equ def of seminorm}). This choice integrates well with Definition \ref{def weak conical metrics}.

Then, the integration by parts proceeds line to line as in 
proof $(i)$ of Theorem 9.9 in \cite{GT}. For the sake of a self-contained proof, and of emphasizing the operator $\nabla^{1,1,\beta}$ we choose, we include the detail here. Denote

\[\Delta_{0,\beta}=\frac{|z|^{2-2\beta}}{\beta^2}\frac{\partial^2}{\partial z\partial \bar{z}},\ \omega_{0,\beta}=\frac{\beta^2}{|z|^{2-2\beta}}\frac{\sqrt{-1}}{2} dz\wedge d\bar{z}.\]
Then \begin{equation}
\Delta=4\Delta_{0,\beta}+\Sigma_{j=1}^{2n-2}\frac{\partial^2}{\partial y_{j}^2}.
\end{equation}
Consider  $A(R)=\underline{B}(R)\times \underline{D}(R)$ as  the polycylinder  as the defintion in \ref{equ def of the annulus}. Let $R$
be large enough such that $A(R)\supset supp f$, then 
\begin{eqnarray}\label{equ first step of L2 estimate integration by parts }& &\int_{\underline{B}(R)}\int_{\underline{D}(R)}(\Delta w_{\epsilon_{k}})^2\omega_{0,\beta}\wedge dy_{1}...\wedge dy_{2n-2}\nonumber
\\&=&16\int_{\underline{B}(R)}\int_{\underline{D}(R)}(\Delta_{0,\beta} w_{\epsilon_{k}})^2\omega_{0,\beta}\wedge dy_{1}...\wedge dy_{2n-2}\nonumber
\\& &+8\Sigma_{j=1}^{2n-2}\int_{\underline{B}(R)}\int_{\underline{D}(R)}(\Delta_{0,\beta}w_{\epsilon_{k}})w_{\epsilon_{k},jj}\omega_{0,\beta}\wedge dy_{1}...\wedge dy_{2n-2}\nonumber
\\& &+\Sigma_{i,j=1}^{2n-2}\int_{\underline{B}(R)}\int_{\underline{D}(R)}w_{\epsilon_{k},ii}w_{\epsilon_{k},jj}\omega_{0,\beta}\wedge dy_{1}...\wedge dy_{2n-2}.
\end{eqnarray}

Ignoring the constant coefficient temporarily, it suffices to deal with the second term above. Since $f\in C^{\infty}_{c}(B)$,  and Donaldson's Schauder estimate in \cite{Don} is smooth in the tangential directions, we have 
 \begin{equation}\label{equ the only 3rd derivative in the L2 estimate:}\frac{\partial}{\partial x_{j}}\Delta_{0,\beta} w_{\epsilon_{k}}\in C^{0}[A(R)].
 \end{equation}
 We will show in the below that it's convenient to do integration by parts over these polycylinders, in our case. 

Using the condition (\ref{equ the only 3rd derivative in the L2 estimate:}) and Lemma 2.5 in \cite{WYQ}, the tangential derivatives $\frac{\partial}{\partial y_{j}}$ can be integrated by parts. Hence

\begin{eqnarray}\label{equ 1 integration by parts with boundary term}& &
\Sigma_{j=1}^{2n-2}\int_{\underline{B}(R)}\int_{\underline{D}(R)}(\Delta_{0,\beta}w_{\epsilon_{k}})w_{\epsilon_{k},jj}\omega_{0,\beta}\wedge dy_{1}...\wedge dy_{2n-2}\nonumber
\\&=& -\Sigma_{j=1}^{2n-2}\int_{\underline{B}(R)}\int_{\underline{D}(R)}(\Delta_{0,\beta}w_{\epsilon_{k},j})w_{\epsilon_{k},j}\omega_{0,\beta}\wedge dy_{1}...\wedge dy_{2n-2}\nonumber
\\& &+ \Sigma_{j=1}^{2n-2}\int_{\partial(\underline{B}(R)\times \underline{D}(R))}(\Delta_{0,\beta}w_{\epsilon_{k}})w_{\epsilon_{k},j}n_{j}\omega_{0,\beta}\wedge dy_{1}...\wedge dy_{2n-2}.
\end{eqnarray}
$w_{\epsilon_{k}}\in C^{2,\alpha,\beta}[A({R})]$ implies the tangential-normal mixed derivatives $\nabla_{0,\beta}w_{\epsilon_{k},j}$ are in 
$L^{\infty}[A(R)]$. Moreover,  $\omega_{\beta}$ is a product metric of $\omega_{0,\beta}$ with the Euclidean metric in the tangential directions along $D$.  Then, again, Lemma 2.5 in \cite{WYQ} and Fubini's Theorem   imply we can integrate the $\Delta_{0,\beta}$ on the first $\C$-slice by parts to obtain 
\begin{eqnarray}\label{equ 2 integration by parts with boundary term}& &
-\Sigma_{j=1}^{2n-2}\int_{\underline{B}(R)}\int_{\underline{D}(R)}(\Delta_{0,\beta}w_{\epsilon_{k},j})w_{\epsilon_{k},j}\omega_{0,\beta}\wedge dy_{1}...\wedge dy_{2n-2}\nonumber
\\&=& \Sigma_{j=1}^{2n-2}\int_{\underline{B}(R)}dy_{1}...\wedge dy_{2n-2}\int_{\underline{D}(R)}|\nabla_{0,\beta}w_{\epsilon_{k},j}|^2\omega_{0,\beta}
\\& &-\Sigma_{j=1}^{2n-2}\int_{\underline{B}(R)}dy_{1}...\wedge dy_{2n-2}\int_{\partial{\underline{D}}(R)}<\nu, \nabla_{0,\beta}w_{\epsilon_{k},j}>_{\beta}w_{\epsilon_{k},j}\omega_{0,\beta},\nonumber
\end{eqnarray}
where $\nu$ is the outer-normal of $\partial{\underline{D}}(R)$ with respect to $\omega_{0,\beta}$. Theorem 1.11 in \cite{ChenWang} and the compactly supported property of $f$ impliy
\begin{equation}\label{equ decay of the newton potential}
|\nabla_{0,\beta}w_{\epsilon_{k},j}| \in O(|x|^{-2n}),\ w_{\epsilon_{k},j}\in O(|x|^{-2n+1}),\ |\Delta_{0,\beta}w_{\epsilon_{k}}| \in O(|x|^{-2n}).
\end{equation}

Thus, combing (\ref{equ 1 integration by parts with boundary term}) and (\ref{equ 2 integration by parts with boundary term}), using (\ref{equ decay of the newton potential}), let $R\rightarrow \infty$,  then the boundary terms all tend to $0$, and we obtain the following as in proof (i) of Theorem 9.9 in \cite{GT}.
\begin{eqnarray}\label{equ conclusion integration by parts with boundary terms}& &
\Sigma_{j=1}^{2n-2}\int_{\C^{n}}(\Delta_{0,\beta}w_{\epsilon_{k}})w_{\epsilon_{k},jj}\omega_{0,\beta}\wedge dy_{1}...\wedge dy_{2n-2}\nonumber
\\&=& \Sigma_{j=1}^{2n-2}\int_{\C^{n}}|\nabla_{0,\beta}w_{\epsilon_{k},j}|^2\omega_{0,\beta}\wedge dy_{1}...\wedge dy_{2n-2}.
\end{eqnarray}

Handling the term $\Sigma_{i,j=1}^{2n-2}\int_{\underline{B}(R)}\int_{\underline{D}(R)}w_{\epsilon_{k},ii}w_{\epsilon_{k},jj}\omega_{0,\beta}\wedge dy_{1}...\wedge dy_{2n-2}$ in (\ref{equ first step of L2 estimate integration by parts }) in the similar and easier way, let $R\rightarrow \infty$, then the boundary terms all tend to $0$, and  we deduce from (\ref{equ first step of L2 estimate integration by parts }) and (\ref{equ conclusion integration by parts with boundary terms}) that 
\begin{eqnarray}\label{equ L2 estimate detail version }& &\int_{\C^n}(\Delta w_{\epsilon_{k}})^2\omega_{0,\beta}\wedge dy_{1}...\wedge dy_{2n-2}\nonumber
\\&=&16\int_{\C^n}(\Delta_{0,\beta} w_{\epsilon_{k}})^2\omega_{0,\beta}\wedge dy_{1}...\wedge dy_{2n-2}\nonumber
\\& &+8\Sigma_{j=1}^{2n-2}\int_{\C^n}|\nabla_{0,\beta}w_{\epsilon_{k},j}|^2\omega_{0,\beta}\wedge dy_{1}...\wedge dy_{2n-2}\nonumber
\\& &+\Sigma_{i,j=1}^{2n-2}\int_{\C^n} |w_{\epsilon_{k},ij}|^2\omega_{0,\beta}\wedge dy_{1}...\wedge dy_{2n-2}.
\end{eqnarray}
Thus identity (\ref{equ W22 seminorm identity}) is true for Newtonian potentials of compactly supported smooth functions. 

Step 2: By Young's inequality, since $\Gamma(x,y), \nabla\Gamma(x,y) \in L^{1}(B)$,  (by Donaldson's work \cite{Don}, also see \cite{ChenWang}), we conclude 
\begin{equation}\label{equ L2 seminorm identity}
|w_{\epsilon_{k_1}}-w_{\epsilon_{k_2}}|_{L^2{(B)}}
\leq |\int_{B}\Gamma(x,y)(f_{\epsilon_{k_1}}(y)-f_{\epsilon_{k_2}}(y))dy|_{L^2{(B)}}\leq |f_{\epsilon_{k_1}}-f_{\epsilon_{k_2}}|_{L^2{(B)}},
\end{equation}
and \begin{eqnarray}\label{equ W12 seminorm identity}& &|\nabla w_{\epsilon_{k_1}}-\nabla w_{\epsilon_{k_2}}|_{L^2{(B)}}\nonumber
\\&\leq & |\int_{B}[\nabla_{x}\Gamma(x,y)](f_{\epsilon_{k_1}}(y)-f_{\epsilon_{k_2}}(y))dy|_{L^2{(B)}}\nonumber
\\&\leq & |f_{\epsilon_{k_1}}-f_{\epsilon_{k_2}}|_{L^2{(B)}}.
\end{eqnarray}

Then, since $f_{\epsilon_{k}}\rightarrow f$ in the $L^2(B)$-sense, 
then $w_{\epsilon_{k}}$ is a Cauchy-Sequence  in $W^{2,2,\beta}(B)-$ space. Thus, by the completeness of the $W^{2,2,\beta}(B)-$space, there exists a $w\in W^{2,2,\beta}(B)$ such that 
\[\lim_{k\rightarrow \infty}|w_{\epsilon_{k}}-w|_{W^{2,2,\beta}(B)}=0.\]
Then, we define $w=N_{\beta,B}f$. By (\ref{equ W12 seminorm identity}), (\ref{equ W22 seminorm identity}), and (\ref{equ L2 seminorm identity}), the proof of Lemma \ref{lem W22 estimate of the newtonian potential} is complete.  
\end{proof}
\begin{rmk} The feature of the $\nabla^{1,1,\beta}$ operator we consider in (\ref{equ W22 seminorm identity}) is: it is just the usual real Hessian in the tangential direction of $D$, it contains all the mixed derivatives. But, in the normal direction of $D$, it only contains the complex $(1,1)$ derivative. This is the one of the main points of this article: with this slightly weaker Hessian, we  obtain  $W^{2,p,\beta}-$estimates for all $p\in (1,\infty)$.  We don't think any $W^{2,p}-$theory for the real Hessian $\nabla^2$ (of $\omega_{\beta}$)  could be true when $p$ is large. 
\end{rmk}
   
\section{The Calderon-Zygmund inequalities. \label{section Calderon Zygmund and potential}}
In this section, we use the  Euclidean metric $\omega_{E}$ in the polar coordinates to define the distances and balls, for the sake of the cube-decomposition.   We show that with the help of Theorem 1.11 in \cite{ChenWang}, the Calderon-Zygmund theory in \cite{CalderonZygmund}
works suprisingly well in the conical setting, after overcoming a technical difficulty. Namely, the main technical difficulty is that $T$
is not selfadjoint. However, as presented below, this difficulty can be 
easily overcomed, by observing that $T^{\star}$ (the dual of $T$) also 
possess similar good properties as the Calderon-Zygmund singular integral operators.  We follow the proof of Theorem 9.9 in \cite{GT}.
\begin{lem}\label{lem T is weakly 1,1}  Let $B$ be a ball with finite radius.    The operator $T$ is weakly-$(1,1)$ bounded i.e  for any $f\in L^{2}(B)$, we have
\begin{equation}
\mu_{Tf}(t)\leq \frac{C}{t}|f|_{L^{1}(B)},
\end{equation}
and \begin{equation}
\mu_{T^{\star}f}(t)\leq \frac{C}{t}|f|_{L^{1}(B)},
\end{equation}
where $C$ only depend on $\beta$ and $n$.
\end{lem}

\begin{proof}{of Lemma \ref{lem T is weakly 1,1}:}
In the polar coordinates, with respect to the Euclidean metric,  we consider a cube $K_0$ (with respect to the Euclidean metric) big enough so that the following holds.  For every $K$ in the first $[\frac{(10n)^{10n}}{\beta}]$ (the smallest integer bigger than $\frac{(10n)^{10n}}{\beta}$) dyadic cut of $K$, $\int_{K}|f|\leq tK$. Exactly as in Theorem 9.9 in \cite{GT}, we consider the dyadic cuts of 
$K_{0}$ subject to $f$ and $t$. Then we obtain cubes $K_{l},\ l=1,2...$ such that 
\begin{equation}\label{equ L1 norm of f on bad cubes}
t|K_{l}|\leq \int_{K_{l}}|f|\leq 2^{2n}t|K_{l}|, \ \textrm{for all}\ l,
\end{equation}
and 
\begin{equation}\label{equ f small almost everywhere away from bad cubes}
f\leq t\ \textrm{almost everywhere over}\ G=K_{0}\setminus F,
\end{equation}
where $F=\cup_{l}K_{l}$.

Then we consider the "good" and "bad" decomposition of $f$ as 
$f=g+b$ such that
\begin{displaymath}b=\left \{
\begin{array}{ccr}& 0,\ \textrm{over}\ G;\\
 & b_{l},\   \frac{1}{|K_{l}|}\int_{K_{l}}b_{l}=0.
 \end{array} \right.
\end{displaymath}
and 
\begin{displaymath}g=\left \{
\begin{array}{ccr}& f,\ \textrm{over}\ G;\\
 & \frac{1}{|K_{l}|}\int_{K_{l}}|f| ,\ \textrm{over}\ K_{l}.
 \end{array} \right.
\end{displaymath}
Thus, (\ref{equ L1 norm of f on bad cubes}) and (\ref{equ f small almost everywhere away from bad cubes}) imply 
\begin{equation}\label{equ good part is bounded}
|g|_{L^{\infty}(K_{0})}\leq 2^{2n}t.
\end{equation}
We define $\mu_{Tf}(t)=m\{x\in K_{0}|f(x)>t\} $. Then 
\begin{equation}
\mu_{Tf}(t)\leq \mu_{Tg}(\frac{t}{2})+\mu_{Tb}(\frac{t}{2}).
\end{equation}
As in the proof Theorem 9.9 in \cite{GT}, by Lemma  \ref{lem W22 estimate of the newtonian potential}, we estimate $\mu_{Tg}(\frac{t}{2})$ as 
\begin{equation}\label{equ estimate of mu of the good part}
\mu_{Tg}
\leq  \frac{4}{t^2}\int_{K_{0}}g^2\leq \frac{2^{2n+2}}{t}\int_{K_{0}}g\leq \frac{2^{2n+2}}{t}\int_{K_{0}}|f|.
\end{equation}
\begin{equation}
T{b_{l}}=\int_{K_{l}}\mathfrak{D}\Gamma(x,y)b_{l}(y)dy
\end{equation}
is well-defined when $x\notin K_{l}$. At this stage, actually for any $\mathfrak{D}\in \mathfrak{T}$, there exists a $\mathfrak{D}^{\prime}\in \mathfrak{M}$ such that 
\begin{equation}
\mathfrak{D}\Gamma(x,y)=\mathfrak{D}^{\prime}\Gamma(x,y).
\end{equation}
This is by the translation invariance of the model metric $\omega_{\beta}$ in the directions tangential to $D$. To see this, for example, take $\mathfrak{D}=\frac{\partial^2}{\partial r_{x}x_{j}}\in \mathfrak{T}$, where both the derivatives act on $x$. Notice that in (\ref{eqn heat kernel formula}), in the $D$-tangential directions, the heat kernel only depends on $|\underline{x}-\underline{y}|$, which means 
\begin{equation}
\frac{\partial^2}{\partial r_{x}\partial x_{j}}\Gamma(x,y)=-\frac{\partial^2}{\partial r_{x}\partial y_{j}}\Gamma(x,y).
\end{equation}
Notice that the biggest feature of $\mathfrak{D}^{\prime}$ is that the two derivatives are distributed to different variables, and Lemma \ref{lem 3rd derivative of Green function bounded} holds for them. 

Hence, 
 using  $\frac{1}{|K_{l}|}\int_{K_{l}}b_{l}=0$, we have 
\begin{equation}\label{equ formula of Tb away from the bad cube}
T{b_{l}}(x)=\int_{K_{l}}[\mathfrak{D}^{\prime}\Gamma(x,y)-\mathfrak{D}^{\prime}\Gamma(x,\bar{y})]b_{l}(y)dy,
\end{equation}
where $\bar{y}$ is the center of $K_{l}$, and 
$x\notin B_{\bar{y}}(\frac{10^{10}D_{l}}{\beta})$, $D_{l}$ is the diameter of $K_{l}$.

Case 1: Suppose $dist(\bar{y},\{z=0\})< \frac{\beta^2dist(\bar{y},x)}{1000}$, then using Lemma   \ref{lem 3rd derivative of Green function bounded}, we have 
\begin{equation}
|T{b_{l}}|\leq  \int_{K_{l}}\frac{C|y-\bar{y}|^{\rho_{0}}}{|x-\bar{y}|^{2n+\rho_{0}}}|b_{l}|dy\leq C D_{l}^{\rho_{0}}\int_{K_{l}}|x-\bar{y}|^{-(2n+\rho_{0})}|b_{l}|dy.
\end{equation}
\begin{eqnarray}
& &\int_{K_0\setminus B_{\bar{y}}(\frac{10^{10}D_{l}}{\beta})}|T b_{l}(x)|dx\nonumber
\\&
\leq & C D_{l}^{\rho_{0}}\int_{K_{l}}|b_{l}|dy
\int_{K_0\setminus B_{\bar{y}}(\frac{10^{10}D_{l}}{\beta})}|x-\bar{y}|^{-(2n+\rho)}dx\nonumber
\\&\leq & C\int_{K_{l}}|b_{l}|dy \{ D_{l}^{\rho_{0}}\nonumber
\int_{|a|\geq D_{l}}|a|^{-(2n+\rho)}da\}
\\&= &  C\int_{K_{l}}|b_{l}|dy.
\end{eqnarray}
Case 2: 
Suppose $dist(\bar{y},\{z=0\})>\frac{ \beta^2 dist(\bar{y},x)}{1000}$, then using Lemma   7.5 in \cite{ChenWang} (with the condition $P(y)\geq \frac{\beta^4|x-y|}{100^{100}}$), we have 
\begin{equation}
|T{b_{l}}|\leq  \int_{K_{l}}\frac{C|y-\bar{y}|}{|x-\widehat{y}_{x}|^{2n+1}}|b_{l}|dy\leq C D_{l}\int_{K_{l}}|x-\widehat{y}_{x}|^{-(2n+1)}|b_{l}|dy,
\end{equation}
where $\widehat{y}_{x}$ is a point in the line segment connecting 
$y$ and $\bar{y}$. 
Thus, 
\begin{eqnarray}\label{equ estimate Tb}
& &\int_{K_0\setminus B_{\bar{y}}(\frac{10^{10}D_{l}}{\beta})}|T b_{l}(x)|dx\nonumber
\\&
\leq & C D_{l}\int_{K_{l}}|b_{l}|dy
\int_{K_0\setminus B_{\bar{y}}(\frac{10^{10}D_{l}}{\beta})}|x-\widehat{y}_{x}|^{-(2n+1)}dx\nonumber
\\&
\leq & C D_{l}\int_{K_{l}}|b_{l}|dy\nonumber
\int_{K_0\setminus B_{\bar{y}}(\frac{10^{10}D_{l}}{\beta})}|x-\bar{y}|^{-(2n+1)}dx
\\&\leq & C\int_{K_{l}}|b_{l}|dy \{ D_{l}
\int_{|a|\geq D_{l}}|a|^{-(2n+1)}da\}\nonumber
\\&= &  C\int_{K_{l}}|b_{l}|dy\leq C\int_{K_{l}}|f|dy.
\end{eqnarray}
\begin{rmk}The reason we have so many $\beta$'s in the proof is that in this section  we use the distance and cubes with respect to the Euclidean metrics $\omega_{E}$ in the polar coordinates, but in \cite{ChenWang} we use the cone distances with respect to $\omega_{\beta}$. Their relation is 
\[\beta\omega_{\beta}\leq \omega_{E}\leq \frac{\omega_{\beta}}{\beta}.\]
We can only consider cubes with respect to the reference metric $\omega_{E}$.
\end{rmk}
Thus, by the argument in page 234 of \cite{GT}, we deduce
\begin{equation}\label{equ estimate of the mu of the bad part}
\mu_{Tb}(f)\leq \frac{C}{t}|f|_{L^{1}(B)}.
\end{equation}

Combining (\ref{equ estimate of the mu of the bad part}) and (\ref{equ estimate of mu of the good part}), the proof of Lemma \ref{lem T is weakly 1,1} on the opeartors $T\in \mathfrak{T}$ is complete. Using exactly the proof above,  the estimate on the adjoint operator $T^{\star}$ follows. Actually we have a slightly shorter proof for $T$ (which does not require dividing the situation into 2 cases). However, since we want a single proof to work for both $T$ and $T^{\star}$, we only present the longer proof above.

\end{proof}

The following lemma is only needed in proving $T^{\star}f$ is densely defined in $L^{p}$, combined with the other results of this article.  Though not fully needed in the proof of the results in the introduction,  we think it has its own interest.  
\begin{lem}\label{lem Tstar maps calpha to itself}$T^{\star}$ is bounded linear map from $C^{\alpha,\beta}$ to 
itself, for all $\alpha<\min\{\frac{1}{\beta}-1,1\}$.
\end{lem}
\begin{proof}{of Lemma \ref{lem Tstar maps calpha to itself}:}

Using Theorem 1.11 in \cite{ChenWang} and  Lemma \ref{lem 3rd derivative of Green function bounded}, the proof is exactly as in Proposition 5.3 in \cite{ChenWang}.

\end{proof}

\begin{Def}\label{def of mixed derivatives acting on different variables} Similar to the definition in (\ref{Def of mathbb T}), we define the mixed derivative operators as \begin{eqnarray*}\mathfrak{M}&=&\{D_{x}D_{y_{j}},\ y_{j}\ \textrm{is a tangential variable to}\ D; D_{y}D_{x_{j}},\ x_{j}
\\& & \textrm{is a tangential variable to}\ D.\}
\end{eqnarray*}
The feature of this set of operators is that the two derivatives are distributed to  different variables. 
\end{Def}
\begin{lem}\label{lem 3rd derivative of Green function bounded}
For any second order spatial derivative operator $\mathfrak{D}\in \mathfrak{M}$. Suppose $\rho_{0}= \min(\frac{1}{\beta}-1,\ 1)$, $|x|=1$ and $|v_1|,\ |v_2|< \frac{1}{8}$, we have
$$|\mathfrak{D}\Gamma(x,v_1)-\mathfrak{D}\Gamma(x,v_2)|\leq C|v_1-v_2|^{\rho_{0}}.$$

\end{lem}
\begin{proof}{of Lemma \ref{lem 3rd derivative of Green function bounded}:} It's an easier version of the arguments in section 8 in \cite{ChenWang}.

\end{proof}

\section{$W^{2,p}$ and $C^{1,\alpha,\beta}$ estimates with $L^{p}$-right hand side. \label{section W2p and C1alpha}}
In this section, we prove Theorem \ref{thm W2p estimate of general L with small oscillation} by proving Theorem \ref{thm W2p estimate of compact supported W2,p functions}, and  we also prove  Theorem \ref{thm C1alpha estimate with Lp right  hand side}. These 3 theorems are the main technical building blocks of the local regularity results in 
Theorem \ref{thm general regularity of conical Monge-Ampere equations} and Corollary \ref{Cor regularity of weak KE metrics}.

\begin{proof}{of Theorem \ref{thm W2p estimate of general L with small oscillation}:} Suppose $L=a^{i\bar{j}}\frac{\partial^2}{\partial z_i\bar{z}_{j}}$, by multiplying a cutoff function $\eta^2$, we compute 
\begin{eqnarray*}& &\Delta_{\beta}(\eta^2 u)
\\&=&(\Delta_{\beta}-L)(\eta^2 u)+L(\eta^2 u)
\\&=&(\Delta_{\beta}-L)(\eta^2 u)
+2Re a^{i\bar{j}}(\eta^2)_{i}(u)_{\bar{j}}+a^{i\bar{j}}(\eta^2)_{i\bar{j}}u+\eta^2 f.
\end{eqnarray*}
Using Theorem \ref{thm W2p estimate of compact supported W2,p functions}, we deduce
\begin{eqnarray*}
& &[\eta^2 u]_{W^{2,p,\beta}, B(2)}
\\&\leq & C|\Delta_{\beta}(\eta^2 u)|_{L^{p}, B(2)}
\\&\leq &C |(\Delta_{\beta}-L)(\eta^2 u)|_{L^{p}, B(2)}
+ C|2Re a^{i\bar{j}}(\eta^2)_{i}(u)_{\bar{j}}|_{L^{p}, B(2)}
\\& &+ C|a^{i\bar{j}}(\eta^2)_{i\bar{j}}u|_{L^{p}, B(2)}
+C |\eta^2 f|_{L^{p}, B(2)}
\\&\leq &C\delta_{0}[\eta^2 u]_{W^{2,p,\beta}, B(2)}+C| u|_{W^{1,p,\beta}, B(2)}+|\eta^2 f|_{L^{p}, B(2)}.
\end{eqnarray*}
Choosing $\delta_{0}\leq \frac{1}{2C}$, and $\eta$ be the cutoff function such that 
\[\eta\equiv 1\ \textrm{over}\ B(1);\ \eta=0\ \textrm{over}\ B(2)\setminus B(\frac{3}{2}),\] the desired conclusion in Theorem \ref{thm W2p estimate of general L with small oscillation}
follows. 

\end{proof}

 \begin{thm}\label{thm W2p estimate of compact supported W2,p functions}
Let $B$ be a ball in $C^{n}$. Then $T$ is a bounded linear map from $L^{p}(B)$ to itself, for $p\in (1,\infty)$. i.e for all $f\in L^{p}(B) $, we have 
\[|T f|_{L^{p}(B)}\leq |f|_{L^{p}(B)}.\]

 Consequently, let $u\in W^{2,p,\beta}_{c}(B)\cap C^{2}(B\setminus D),\ p\in [2,\infty)$, then 
\begin{equation*}
[u]_{W^{2,p,\beta}(B)}\leq C |\Delta_{\beta} u|_{L^p(B)},
\end{equation*}
where $C$ only depend on $n,\beta,p$.
\end{thm}

\begin{proof}{of Theorem \ref{thm W2p estimate of compact supported W2,p functions}:} 
    By Marcinkiewicz-intepolation in Theorem 9.8 in \cite{GT}, Lemma \ref{lem W22 estimate of the newtonian potential},  and Lemma \ref{lem T is weakly 1,1}, we deduce both $T$ and $T^{\star}$ are bounded linear map 
    from $L^{p}$ to $L^{p}$, $1<p\leq 2$ i.e 
    \begin{equation}\label{equ pp bound of T p between 1 and 2}
    |Tf|_{p,B}\leq C|f|_{p,B},\ \textrm{for all}\ 1<p\leq 2,
    \end{equation}
   and
 \begin{equation}
    |T^{\star}f|_{p,B}\leq C|f|_{p,B},\ \textrm{for all}\ 1<p\leq 2.
    \end{equation}
Then for $p>2$, we conclude  for all $f, g\in C_{c}^{\infty}(B)$  that
\begin{eqnarray}\label{eqn pp bound for T}
& & |Tf|_{p,B}\nonumber
\\&=& \sup_{|g|_{p^{\prime},B}=1}\int_{B}(Tf)gdx\nonumber
=\sup_{|g|_{p^{\prime},B}=1}\int_{B}f(T^{\star}g)dy\nonumber
\\&\leq &\sup_{|g|_{p^{\prime},B}=1} |f|_{p,B}|T^{\star}g|_{p^{\prime},B}\ \ (1<p^{\prime}<2)\nonumber
\\&\leq  & C\sup_{|g|_{p^{\prime},B}=1} |f|_{p,B}|g|_{p^{\prime},B}\nonumber
\\&=& C|f|_{p,B}.
\end{eqnarray}
Notice that $u=N_{\beta,\C^{n}}(\Delta_{\beta}u)$, by Lemma \ref{lem Green's representation}. Then, combining   (\ref{eqn pp bound for T}),  (\ref{equ pp bound of T p between 1 and 2}),  the fact that $C^{\infty}_{c}(B)$ is dense in $L^{p}(B)$,  and the Laplace equation 
\[\frac{|z|^{2-2\beta}}{\beta^2}\frac{\partial^2u}{\partial z\partial \bar{z}}=\Delta_{\beta} u-\Sigma_{i=1}^{n-1}\frac{\partial^2u}{\partial w_{i}\partial \bar{w}_{i}},\]
we obtain the estimates in all directions are obtained.  The  proof of Theorem \ref{thm W2p estimate of compact supported W2,p functions} is complete. Moreover, we've shown 
\begin{equation*}
|T^{\star}g|_{q,B}\leq C|g|_{q,B},\ \textrm{for all}\ 1<q< \infty. 
\end{equation*}
\end{proof}
\begin{cor}\label{cor regularity of weak solution to laplace equation with Lp righthand side}
Suppose $u\in W^{1,2,\beta}[B(2)]$ is a weak-solution  to \[\Delta u=f,\  f\in L^{p}[B(1)],\ \infty>p\geq 2.\]
Then $u$ is actually in $W^{2,p,\beta}[B(\frac{1}{2})]$ and is therefore a strong-solution to the above equation in $B(\frac{1}{2})$.
\end{cor}
\begin{proof}{of Corollary \ref{cor regularity of weak solution to laplace equation with Lp righthand side}:}
The proof is quite straight forward. Just notice $N_{\beta,B(1)}f\in W^{2,p,\beta}[B(1)]$, and $v=u-N_{\beta,B(1)}f\in W^{2,p,\beta}[B(1)]$ is a weak solution to the harmonic equation
\[\Delta v=0,\  \textrm{over}\ B(1).\]
Thus by Lemma 2.1 in \cite{WYQ}, $v\in C^{2,\alpha,\beta}[B(\frac{2}{3})]$. Thus $u=v+N_{\beta,B(1)}f\in  W^{2,p,\beta}[B(\frac{1}{2})]$.

\end{proof}

\begin{proof}{of Theorem \ref{thm C1alpha estimate with Lp right  hand side}:} This is an easier version of the work in \cite{ChenWang}. By Corollary \ref{cor regularity of weak solution to laplace equation with Lp righthand side} and Donaldson's Schauder-estimate in \cite{Don}, it suffices to estimate the Newtonian potential of $f$:
\[N_{\beta}f=\int_{B(1)}\Gamma(x,y)f(y)dy.\]

By Lemma 7.2 in \cite{ChenWang}, we estimate 
\begin{eqnarray}
& &|\nabla (\Gamma \star f)|=|\int_{B(1)}(\nabla \Gamma (x,y))f(y)dy|\nonumber
\\&\leq &C \int_{B(1)}\frac{1}{|x-y|^{2n-1}}  |f(y)|dy\nonumber
\\&\leq & C|\frac{1}{|x-y|^{2n-1}}|_{L^{p^{\prime}},B(1)}|f|_{L^{p},B(1)}.
\end{eqnarray}
Since $p>2n$, we have $p^{\prime}<\frac{2n}{2n-1}$. Then 
\[|\frac{1}{|x-y|^{2n-1}}|_{L^{p^{\prime}},B(1)}\leq C,\]
and 
\[|\nabla (\Gamma \star f)|_{C^{0}[B(\frac{1}{2})]}\leq C|f|_{L^{p},B(1)}.\]
Next we estimate the H\"older norm of $\nabla (\Gamma \star f)$. Without loss of generality, we assume $|x_1|=\delta$ and $x_2=0$, which is the main issue. The H\"older estimate for all general $x_{1},x_{2}$ follows from the proof of Proposition 5.3 in \cite{ChenWang}.  We compute
\begin{eqnarray}
& &|\nabla (\Gamma \star f)(x_1)-\nabla (\Gamma \star f)(0)|\nonumber
\\&=&|\int_{B(1)}[\nabla \Gamma (x_1,y)-\nabla \Gamma (0,y)]f(y)dy|\nonumber
\\&\leq  & I_1+I_2,
\end{eqnarray}
where \[I_1=|\int_{B(1)\cap \{|y|> 10\delta\}}[\nabla \Gamma (x_1,y)-\nabla \Gamma (0,y)]f(y)dy|\]
and 
 \[I_2=|\int_{B(1)\cap \{|y|\leq  10\delta\}}[\nabla \Gamma (x_1,y)-\nabla \Gamma (0,y)]f(y)dy|.\]
 Then it's obvious that 
 \begin{eqnarray}
 & &I_2\nonumber
 \\&\leq &|\int_{ \{|y|\leq  10\delta\}}\nabla \Gamma (x_1,y)f(y)dy|+|\int_{\{|y|\leq  10\delta\}}\nabla \Gamma (0,y)f(y)dy| \nonumber
 \\&\leq & C\int_{\{|y|\leq  10\delta\}}\frac{1}{|x_1-y|^{2n-1}}|f(y)|dy+C\int_{ \{|y|\leq  10\delta\}} \frac{1}{|y|^{2n-1}}|f(y)|dy\nonumber
 \\&\leq & C\delta^{\alpha_0}|f|_{p},
 \end{eqnarray}
 where $\alpha_{0}=1-\frac{2n}{p}$.
 For the estimate of $I_1$, we should assume $$2n<p<\frac{2n}{1-\min \{\frac{1}{\beta}-1; 1\}}\ (\textrm{if}\ \beta\leq \frac{1}{2}\ 
 \textrm{we just assume}\ 2n<p<\infty).$$
 Thus $1-\frac{2n}{p}<\min\{\frac{1}{\beta}-1,1\}$. This does not change the conclusion of Theorem \ref{thm C1alpha estimate with Lp right  hand side}, because what we assume is an additional upper bound on $p$. Next, we estimate $I_{1}$. By Lemma 8.2 in \cite{ChenWang}, we compute
  \begin{eqnarray}
 & &I_1\nonumber
 \\&\leq &\int_{ \{|y|\geq  10\delta\}}\frac{1}{|y|^{2n-1}}|\nabla \Gamma (\frac{x_1}{|y|},\frac{y}{|y|})-\nabla \Gamma (0,\frac{y}{|y|})||f(y)|dy 
 \nonumber
 \\&\leq & C\delta^{\alpha_0+\epsilon}\int_{\{|y|\geq  10\delta\}}\frac{1}{|y|^{2n-1+\alpha_{0}+\epsilon}}|f(y)|dy \nonumber
 \\&\leq & C\delta^{\alpha_0}|f|_{p} 
 \end{eqnarray}
 where $\alpha_0 =1-\frac{2n}{p}$, and $\epsilon$ is chosen such that 
 $\alpha_{0}+\epsilon<\min\{\frac{1}{\beta}-1,1\}$. 
\end{proof}

\section{KRF metrics with small ossilations. \label{section KRF metrics with small ossilations}}

In $\C\times \C^{n-1}$,  consider the standard conical K\"ahler-Ricci flat metric with cone angle $\beta\in (0,1) $ along the divisor $\{0\}\times \mathbb{C}^{n-1}$.
\[
\omega_{\beta} = {\beta^2\over {|z|^{2-2\beta}}} |d\,z|^2 + |d\, w|^2,
\]
where $z\in \C$ and $w \in \C^{n-1}.\;$  We say a complex linear transformation $L$ splits along $D$, if the first component $\C\times \{0\}$ in $\C\times \C^{n-1}$ is an invariant space of  $L$, and the tangential component $\{0\}\times \C^{n-1}$ is also an invariant space of  $L$.   In this section,  we prove the following regularity proposition, which is crucial to establish Theorem
\ref{thm general regularity of conical Monge-Ampere equations}.

\begin{prop}\label{prop small ossilation implies regularity of RF metrics}
   Suppose $L$ is a linear transformation which splits along $D$, and $(L^\star\omega_{\beta})^{n}=\omega_{\beta}^{n}$. Then there exists a constant $Q_{0}$ depending on $\beta,\ n$, and the supremum of eigenvalues of $LL^{\star}$ with the following properties.   Suppose  $\phi$ is a 
   pluri-subharmonic function which satisfies 
   \begin{itemize}
   \item $\omega^{n}_{\phi}=e^{f}\omega_{\beta}^n$, $f\in C^{1,1,\beta}[B(100)]$;
   \item $(1-\delta)L^{\star}\omega_{\beta}\leq \omega_{\phi}\leq (1+\delta)L^{\star}\omega_{\beta}$, where $\delta< <1$ is sufficiently small with respect to the supremum of eigenvalues of $LL^{\star}$.
   \end{itemize}
   Then $\phi \in C^{2,\alpha,\beta}[B(\frac{1}{Q_{0}})],\  \textrm{for all}\
\alpha<\min\{\frac{1}{\beta}-1,1\}.$

 In particular, suppose  

\begin{equation}\label{equ small ossilation condition}
 \omega_\phi^n = \omega_{\beta}^n,\  (1-\delta)\omega_{\beta}\leq \omega_\phi\leq (1+\delta)\omega_\beta\ \textrm{over}\ B(100)
\end{equation}
for $\delta<<1$ small enough, then $\phi \in C^{2,\alpha,\beta}[B(\frac{1}{2})],\  \textrm{for all}\
\alpha<\min\{\frac{1}{\beta}-1,1\}.$
\end{prop}
\begin{proof}{of Proposition \ref{prop small ossilation implies regularity of RF metrics}:}  We only prove  the second part on the special case (\ref{equ small ossilation condition}), the proof of the general case is the same. 

 Using (\ref{equ small ossilation condition}) and Proposition \ref{prop Solvability of the ddbar equation}, over $B(10)$, we can choose a potential, still denoted as $\phi$, such that 
   \begin{equation}\label{equ choosing a good potential on B2}
   |\phi|_{0,B(10)}\leq C(n,\beta),
   \end{equation}
   where $C(n,\beta)$ is a constant which only depends on the dimension $n$ and angle $\beta$.
   
For any unit vector $v \in \{0\}\times \C^{n-1}.\;$ tangential to the divisor,  and for any small positive constant $\epsilon > 0$, define difference quotient as
\[
(D_{\epsilon, v}  \phi) (z,w) = {{ \phi(z, w+\epsilon \cdot v) - \phi(z,w)}\over \epsilon}.
\]
Let $\epsilon \rightarrow 0,$ we have
\[
\displaystyle \lim_{\epsilon \rightarrow 0} D_{\epsilon, v}  \phi  = (\nabla \phi, v).
\]
By (\ref{equ small ossilation condition}), we end up with a trivial but important fact 
\begin{equation}
|\Delta_{\beta}\phi|\leq C\ \textrm{in }\ B(10).
\end{equation}

Using Theorem \ref{thm W2p estimate of general L with small oscillation}, (\ref{equ choosing a good potential on B2}), Corollary \ref{cor regularity of weak solution to laplace equation with Lp righthand side}, and intepolations in the appendix of \cite{ChenWang}, we obtain
\[|\phi|_{W^{2,p,\beta}(B(5))}\leq C,\ \textrm{for all }\ 1<p<\infty.\]
This implies
\[|\frac{\partial \phi}{\partial v}|_{W^{1,p}(B(5))}\leq C,\ \textrm{for all }\ 1<p<\infty.\]

Then, by Lemma 7.23 in \cite{GT}, we conclude the following estimate on the tangential difference quotients
\begin{equation}\label{equ tangential difference quotient estimate} |D_{\epsilon, v}  \phi |_{W^{1,p}B(4)}\leq C. 
\end{equation} 

Therefore, take $D_{\epsilon, v}$ to both hand sides of the Ricci-flat  equation 
\begin{equation}
\omega^{n}_{\phi}=\omega^{n}_{\beta},
\end{equation}
we obtain
\[
\triangle_{\epsilon, v} D_{\epsilon, v}  \phi  = 0,  \ \textrm{over}\ B(9),
\]
where 
\[\triangle_{\epsilon, v}(x)=\int^{1}_{0}[s\phi(x+\epsilon v)+(1-s)\phi(x)]^{i\bar{j}}ds\frac{\partial^2}{\partial z_{i}\partial \bar{z}_{j}}.\]
(\ref{equ small ossilation condition}) implies directly the following.
\[
|\triangle_\beta - \triangle_{\epsilon, v}| \leq \delta ( \omega_\beta )^{-1}. 
\]
 By Evans-Krylov Theorem away from $D$, we have    $\phi \in C^{2}(\C\times C^{n-1}\setminus D) $ (actually $\phi$ is smooth away from $D$, but $C^{2}$ of $\phi$ is all we need here). Then the aprori estimate in Theorem \ref{thm W2p estimate of compact supported W2,p functions} and \ref{thm W2p estimate of general L with small oscillation} are directly applicable. 
Applying Theorem \ref{thm W2p estimate of general L with small oscillation} and (\ref{equ tangential difference quotient estimate}), we have
\[
[D_{\epsilon, v} \phi]_{W^{2,p,\beta},B(2)}\leq C (|D_{\epsilon, v} \phi|_{W^{1,p},B(4)})\leq C.
\]
Since the above holds for all $p$, then applying Theorem (\ref{thm C1alpha estimate with Lp right  hand side}), and the again the lower order estimate (\ref{equ tangential difference quotient estimate}),  we obtain the crucial estimate
\[
|D_{\epsilon, v} \phi|_{1,\alpha,\beta,B(1)}\leq C,\ \textrm{for all}\
\alpha<\min\{\frac{1}{\beta}-1,1\}.
\]
Now, let $\epsilon \rightarrow 0$, since $\phi$ is smooth away from $D$, we have $$\frac{\partial \phi}{\partial v}\in C^{1,\alpha,\beta}[B(\frac{1}{2})]$$ and 
\begin{equation}
|\frac{\partial \phi}{\partial v}|_{1,\alpha,\beta,B(\frac{1}{2})}\leq C,\ \textrm{for all}\
\alpha<\min\{\frac{1}{\beta}-1,1\}.
\end{equation}
This means the mixed derivatives $\frac{\partial^2 \phi}{\partial r\partial w_{i}},\ \frac{1}{r}\frac{\partial^2 \phi}{\partial \theta\partial w_{i}}$, and the pure tangential derivatives $\frac{\partial^2 \phi}{\partial w_{i}\partial w_{j}}$, are all bounded in $C^{1,\alpha,\beta}[B(\frac{1}{2})]$-norm  by $C$ whose dependence is as in Proposition \ref{prop small ossilation implies regularity of RF metrics}.

Using the equation (\ref{equ KRF equation}) and the quasi-isometric condition (\ref{equ small ossilation condition}), exactly as in the proof of Theorem 10.1 in \cite{ChenWanglongtime}, we deduce the crucial
normal-$(1,1)$ derivative.
\[|(\frac{\partial^2}{\partial r^2}+\frac{1}{r}\frac{\partial}{\partial r}+\frac{1}{\beta^2r^2}\frac{\partial^2}{\partial \theta^2})\phi|_{\alpha,\beta,B(\frac{1}{2})}\leq C.\]
  The above  implies our final conclusion
   $$\phi \in C^{2,\alpha,\beta}[B(\frac{1}{2})]\ \textrm{and}\
   |\phi|_{2,\alpha,\beta,B(\frac{1}{2})}\leq C, \ \textrm{for all}\
\alpha<\min\{\frac{1}{\beta}-1,1\}.$$
\end{proof}

Denote $A_{R}$ as the polycylinder of radius $R$. To be precise, we define
\begin{equation}\label{equ def of the annulus} A_{R}=D_{R}\times B_{R},
\end{equation}
where $D(R)$ is the disk with radius $R$ in the $z$-component of $\C^{n}$ (centered at the origin), and $B_{R}$
is the ball with radius $R$ in $D=\{0\}\times \C^{n-1}$ (also centered at the origin). The following Proposition is important.
\begin{prop}\label{prop Solvability of the ddbar equation}There exists a constant $C$ depending on $\beta$ and $n$ with the following properties. Given the equation 
\begin{equation}\label{ddbar equation}\sqrt{-1}\partial \bar{\partial}v=\eta\ \textrm{over}\ A_{1000},
\end{equation}
where $\eta$ is a closed (1,1)-form such that $\eta=\sqrt{-1}\partial\bar{\partial}\phi_{\eta}$ for some $\phi_{\eta}\in C^{1,1,\beta}$. Then there exists a solution $v$ in 
$W^{2,p,\beta}$ (for any $p$) such that 
$$|v|_{W^{2,p,\beta},A_{1}}+|v|_{0,A_{1}}\leq C|\eta|_{L^{\beta,\infty},A_{1000}}  .$$

\end{prop}
\begin{proof}{of Proposition \ref{prop Solvability of the ddbar equation}:}

The proof is exactly as in Proposition 4.1 in \cite{ChenWanglongtime}. Just notice when $\eta$ is merely in $ L^{\infty,\beta}$, the orbifold trick in Lemma 4.3 of \cite{ChenWanglongtime} works the same. Then pulling back upstairs we still obtain a solution by Lemma \ref{lem dbar solvability smooth case}. Hence, take average of this upstairs solution over the discrete orbit of the monodromy group, and push this average down as in Lemma 4.4 in \cite{ChenWanglongtime}, we obtained the solution $v$ we want. 
\end{proof}
\section{Proof of the Main Theorems. \label{section Proof of the Main Theorems}}
In this  section we prove Theorem \ref{thm Liouville} and  \ref{thm general regularity of conical Monge-Ampere equations}. These proofs summarizes the work done in this article.  Corollary \ref{Cor regularity of weak KE metrics} is directly implied by Theorem \ref{thm general regularity of conical Monge-Ampere equations}, by Definition \ref{def local formula for conical KE metrics}.
\begin{proof}{of Theorem \ref{thm Liouville}:} It suffices to show that
(\ref{equ KRF equation}) already implies $\omega$ is $C^{\alpha,\beta}$, then Theorem 1.14 in \cite{ChenWanglongtime} implies Theorem \ref{thm Liouville}. The $C^{\alpha,\beta}$-regularity of the weak conical metric $\omega$ in  Theorem \ref{thm Liouville} is the main work of this article. This can be divided into 2 steps.

Step 1: 7 important results  in \cite{ChenWanglongtime}  directly work in the our weak conical case.  These 7  results are
\begin{itemize}
\item Lemma 6.1 on bounded weakly subharmonic functions in \cite{ChenWanglongtime} (directly works when $\omega$ is merely a weak conical metric );
\item Theorem  6.2 on weak-maximal principles  in \cite{ChenWanglongtime}(directly works when $\omega$ is merely a weak conical metric );
\item Theorem  7.3 and 7.4 on solvability of Dirichilet boundary value problems in \cite{ChenWanglongtime} (directly works when $\omega$ is merely a weak conical metric );
\item Theorem  8.1 on strong-maximal principles  in \cite{ChenWanglongtime} (directly works when $\omega$ is merely a weak conical metric );
\item Lemma   13.1 on Trudinger's harnack inequality  in \cite{ChenWanglongtime} (directly works when $\omega$ is merely a weak conical metric );
\item Proposition 4.1 in \cite{ChenWanglongtime} on solvability of Poincare-Lelong equation with
$C^{\alpha,\beta}$ right hand side. This is substituted by Propsition
\ref{prop Solvability of the ddbar equation} on solvability of Poincare-Lelong equation with $L^{\infty,\beta}$-right hand side, with almost the same proof. 
\end{itemize}
The above 7 results  imply any weak conical metric $\omega$ 
satisfying the conditions in Theorem \ref{thm Liouville} is either linearly-isometric to $\omega_{\beta}$ or admits a tangent cone which is linearly-isometric to $\omega_{\beta}$.

Step 2:  the last paragraph in Step 1 means the second assumption in Theorem 5.1 in \cite{ChenWanglongtime} is fulfilled. Then Theorem 5.1 in \cite{ChenWanglongtime} implies Theorem \ref{thm Liouville}, provided 
we can show $\omega$ is in $C^{\alpha,\beta}$. This is precisely what Proposition \ref{prop small ossilation implies regularity of RF metrics} says.  Actually, Proposition \ref{prop small ossilation implies regularity of RF metrics} is really the main technical result of this article. 

\end{proof}



\begin{proof}{of Theorem \ref{thm general regularity of conical Monge-Ampere equations}:} This theorem is   an interior regularity result, and away from $D$ 
the regularity automatically follows from Proposition 16 of \cite{CDS3}. Thus,  without loss of generality, we assume $\Omega=B_{0}(1)$  (the unit ball centered at the origin). 
This proof is a simple combination of Proposition \ref{prop small ossilation implies regularity of RF metrics}, Theorem \ref{thm Liouville}, and the Chen-Donaldson-Sun's trick in     the proof of Proposition 26 in \cite{CDS2}.

  We consider the rescaling of the metrics and potential as 
  \begin{equation}
\phi_{\lambda}=\lambda^{2}\phi,\ \omega_{\lambda}=\lambda^{2}\omega,\ \widehat{\omega}_{\beta}=\lambda^{2}\omega_{\beta},
\end{equation}
and the rescaling of the coordinates as 
 \begin{equation}\label{equ rescaled coordinates}
\widehat{z}=\lambda^{\frac{1}{\beta}}z,\ \widehat{y}_{j}=\lambda y_{j},\ 1\leq j\leq 2n-2.
\end{equation}

Then the $\widehat{\omega}_{\beta}$ is the model cone metric in 
the coordinates in (\ref{equ rescaled coordinates}). Then equation (\ref{equ conical Monge Ampere}) is rescaled to the following geometric 
equation 
\begin{equation}\label{equ MA equ of omega lambda}
\omega_{\lambda}^{n}=\frac{\widehat{f}}{\beta^{2}}\widehat{\omega}^{n}_{\beta}
\end{equation}
in the coordinates of (\ref{equ rescaled coordinates}), where $\widehat{f}$ is the pulled back function under the coordinate change. Since $\widehat{f}\in C^{1,1,\beta}B_{0}(\lambda)$, by the usual Evans-Krylov Theorem away from $D$, we deduce 
$$|\omega_{\lambda}|_{C^{\alpha}[B(R)\setminus T_{\epsilon}(D)]}\leq C(R,\epsilon),\ \textrm{for all}\ R\leq \frac{\lambda}{2}.$$

Since $f\in C^{\alpha}[B(1)]$ before rescaling, then $\lim_{\lambda\rightarrow \infty} \widehat{f}= Const$ in the sense of $C^{\widehat{\alpha}}$, for all $0<\widehat{\alpha}<\alpha$. Without loss of generality we can assume 
 $$\lim_{\lambda\rightarrow \infty} \frac{\widehat{f}}{\beta^2}= 1. $$

Then, $\omega_{\lambda}$ converges to $\omega_{\infty}$ uniformly over any fixed $B(R)\setminus T_{\epsilon}(D)$ such that 
\begin{equation}\omega^{n}_{\infty}=\widehat{\omega}^n_{\beta}\ \textrm{over}\ \C^{n}\setminus D,
\end{equation} 
and 
\begin{equation}\label{equ quasi isometric condition of the rescaled limit}
\frac{\widehat{\omega}_{\beta}}{C}\leq   \omega_{\infty}\leq C\widehat{\omega}_{\beta}.
\end{equation}
To prove $\omega_{\infty}$ is a weak conical metric in the sense of 
Definition \ref{def weak conical metrics}, it suffices to show 
$\omega_{\infty}$ admits a $C^{\alpha}$-potential near any $p\in D$.  By the proof of the Harnack inequality in item 2 of Lemma 6.1 in \cite{ChenWanglongtime}, and the quasi-isometric condition (\ref{equ quasi isometric condition of the rescaled limit}), it sufficies to 
show 
$\omega_{\infty}$ admits a $L^{\infty}$-potential near any $p\in D$.
This is done simply by applying Proposition \ref{prop Solvability of the ddbar equation} to $\omega_{\lambda}$. Namely,  using the quasi-isometric condition 
\begin{equation}\label{equ quasi isometric condition of omegalambda}
\frac{\widehat{\omega}_{\beta}}{C}\leq   \omega_{\lambda}\leq C\widehat{\omega}_{\beta},
\end{equation}
Proposition \ref{prop Solvability of the ddbar equation} and Theorem \ref{thm C1alpha estimate with Lp right  hand side} imply  for any $p\in D$ in the rescaled coordinates (\ref{equ rescaled coordinates}), when $\lambda$ is sufficiently large, there exists a potential $\phi_{p,\lambda}$ defined in $B_{p}(10^{10})$ 
such that 
\begin{equation}
|\phi_{p,\lambda}|_{C^{\alpha}[B(1)]}\leq C, \ \omega_{\lambda} = \sqrt{-1}\partial \bar \partial\phi_{p,\lambda}.
\end{equation}
Thus, for any $\widehat{\alpha}<\alpha$, $\phi_{p,\lambda}$ converges in $C^{\widehat{\alpha}}[B(\frac{1}{2})]$-topology to $\phi_{p,\infty}\in C^{\widehat{\alpha}}[B(\frac{1}{2})]$ such that 
\begin{equation}
\omega_{\infty} = \sqrt{-1}\partial \bar \partial\phi_{p,\infty}\ \textrm{over}\ B_{p}(\frac{1}{2}),\ \textrm{in the sense of current}.
\end{equation}
Then $\omega_{\infty}$ is a weak conical metric in the sense of Definition \ref{def weak conical metrics}.

Therefore, by Theorem \ref{thm Liouville}, we deduce 
\[\omega_{\infty}=L^{\star}\widehat{\omega}_{\beta},\]
where $L$ is a linear transformation preserving $D$. By the uniform convergence of $\omega_{\infty}$ over any fixed $B(R)\setminus T_{\epsilon}(D)$, and the   the proof of Proposition 2.5 in \cite{CDS1}, 
we deduce  
   \begin{equation}\label{equ L2 convergence to canonical metric}\lim_{\lambda\rightarrow \infty}|\omega_{\lambda}-L^{\star}\widehat{\omega}_{\beta}|_{L^{2}(B_{0}(1))}=0.
   \end{equation}
   
 To modify the convergence in (\ref{equ L2 convergence to canonical metric}) to pointwise convergence, we use the assumption that 
 $f  \in C^{1,1,\beta}(B). $
 
 Since $(L^{\star}\widehat{\omega}_{\beta})^n=\widehat{\omega}_{\beta}^n$, we translate equation (\ref{equ MA equ of omega lambda}) to be 
 \begin{equation}\label{equ translated MA equ of omega lambda}
\omega_{\lambda}^{n}=\frac{\widehat{f}}{\beta^{2}}(L^{\star}\widehat{\omega}_{\beta})^{n}. 
\end{equation}
 
 By Yau's Bochner technique and $h  \in C^{1,1,\beta}[B(1)]$,  we deduce for any $\delta>0$ that 
 \begin{equation}\label{equ subharmonicity of the trace }
 \Delta_{L^{\star}\widehat{\omega}_{\beta}}(tr_{L^{\star}\widehat{\omega}
 _{\beta}}\omega_{\lambda}-n+\delta)\geq -\frac{[h]_{C^{1,1,\beta}[B(1)]}}{\lambda^2}\rightarrow 0\ \textrm{in}\ B(1).
 \end{equation}
Then, (\ref{equ subharmonicity of the trace }), (\ref{equ L2 convergence to canonical metric}), and the Moser's iteration ( as in     the proof of Proposition 26 in \cite{CDS2}) imply 
   \[\lim_{\lambda}|\omega_{\lambda}-L^{\star}\widehat{\omega}_{\beta}|_{L^{\infty,\beta}, (B_{0}(\frac{1}{2}))}=0.\]

Let  $\delta_{0}$  be  small enough with respect to the $\delta$ in 
Proposition \ref{prop small ossilation implies regularity of RF metrics} and the quasi-isometric constant of $\omega_{\phi}$ with respect to $\omega_{\beta}$ in the original coordinates,  there exists a  $\lambda_{0}$ such that for all $\lambda\geq \lambda_{0}$, we have 
\begin{equation}\label{equ pointwise closeness between omega lambda and omega beta} |\omega_{\lambda}-L^{\star}\widehat{\omega}_{\beta}|_{L^{\infty,\beta}(B_{0}(\frac{1}{2}))}<\delta_{0}.
\end{equation}
 Since   (\ref{equ pointwise closeness between omega lambda and omega beta}) implies the following 
crucial small ossilation estimate before rescaling,
\begin{equation}
 |\omega_{\phi}-L^{\star}\omega_{\beta}|_{L^{\infty,\beta}(B_{0}(\frac{1}{2\lambda_{0}}))}<\delta_{0},
\end{equation}
then Proposition  \ref{prop small ossilation implies regularity of RF metrics} implies $\omega_{\phi}\in C^{\alpha,\beta}(\frac{1}{2Q\lambda_{0}}))$, where $Q$ is a constant which only depends on the quasi-isometric constant of $\omega_{\phi}$ with respect to $\omega_{\beta}$ in the original coordinates. The proof of Theorem \ref{thm general regularity of conical Monge-Ampere equations} is complete.
\end{proof}
\section{Appendix A: Poincare-Lelong equation in the smooth case.}
The following lemma is necessary for the results in  \cite{ChenWanglongtime} and also in this article (in the proof of of Proposition \ref{prop Solvability of the ddbar equation}). We believe it's well known to experts, but for the sake of being self-contained we still would like to give a proof here. The proof is actually a simple 
combination of the proof of the Lemma in page 387 of \cite{GH}, and Hormander's results. 
\begin{lem}\label{lem dbar solvability smooth case}
There exists a constant $C$ depending on  $n$ and $p$ with the following properties. Suppose  $\sigma\in L^{2}(B_{10})$ is a closed (1,1)-form such that
\begin{itemize} 
\item $\sigma=\sqrt{-1}\partial\bar{\partial}\phi_{\sigma}$ for some $\phi_{\sigma}\in C^{\alpha}$.
\item $\sigma\in L^{\infty}(B)\cap C^{\alpha}(B_{10}\setminus D)$. 
\end{itemize}  Then there exists a solution $\varphi$ in 
$W^{2,p}$ (for any $0<p<\infty$) to 
\begin{equation}\label{equ ddbar equation smooth}\sqrt{-1}\partial \bar{\partial}\varphi=\sigma\ \textrm{over}\ B_{1},
\end{equation} 
such that 
$$|\varphi|_{W^{2,p},B_{1}}+|\varphi|_{0,B_{1}}\leq C|\sigma|_{L^{\infty},B_{10}}  .$$
\end{lem}
\begin{proof}{of Lemma \ref{lem dbar solvability smooth case}:}
The two conditions of $\sigma$ imply $\sigma \in L^{\infty}(B)$ as a distribution. 

By Hormander's $\bar{\partial}$-solvability in Theorem 2.2.1  \cite{Hormander}, there exists a (1,0)-form $\eta\in L^{2}(B(9))$ such that 
\begin{equation}
\sigma=-\sqrt{-1}\bar{\partial}\eta. 
\end{equation} 

Then, since  $$\sqrt{-1}\bar{\partial}(\partial \eta)=\partial \sigma =0,$$
then $\partial \eta$ is a closed holomorphic $(2,0)$-current. 
Thus, by the regularity of closed holomorphic $(2,0)-$forms,  $\partial \eta$ is actually a smooth holomorphic $(2,0)$-form. By the $d$-Poincare Lemma for smooth holomorphic $(p,0)$-forms, there exists a holomorphic $(1,0)$-form $\xi$ such that
\[\partial \eta=\partial \xi.\]
Thus, we deduce 
\begin{equation}
\partial (\eta-\xi)=0.
\end{equation}
By the conjugate of $\bar{\partial}$-solvability in Theorem 4.2.5 in page 86 of  Hormander's book \cite{Hormanderbook}, we end up with $\partial$-solvability and therefore a form $\gamma$ such that 
\begin{equation}
\partial \gamma= \eta-\xi.
\end{equation}
Then $\sqrt{-1}\partial\bar{\partial}\gamma=\sigma$. Let 
$$\varphi=\frac{1}{2}(\gamma+\bar{\gamma}),$$
then $\varphi$ is real and $\sqrt{-1}\partial\bar{\partial}\varphi=\sigma$.
   Since $\varphi\in W^{1,2}$, then $\varphi$ is a weak solution to 
   \[\Delta \varphi=tr\sigma\in L^{\infty}(B_{5}).\]
   Then, $\varphi$ is a strong solution to the above equation in the
   sense of Chap 9 in \cite{GT}. Then, the estimate in Lemma \ref{lem dbar solvability smooth case} follows from Theorem 9.11 in \cite{GT} and the Moser's iteration.
\end{proof}
\section{Appendix B: An alternative approach to Corollary \ref{Cor regularity of weak KE metrics} by the conical K\"ahler-Ricci flow.}
In this section, we present a short proof of Corollary \ref{Cor regularity of weak KE metrics} when the weak conical K\"ahler-Einstein metric lives on a closed K\"ahler manifold. This proof, while lives on a closed manifold, does not require the $W^{2,p,\beta}$-estimate established in sections \ref{section W2p and C1alpha} and \ref{section Calderon Zygmund and potential}.

Let $(M,[\omega_{0}])$ be a smooth K\"ahler manifold, $D$ be a smooth divisor, $S$ be the defining section of $D$, and $|\ \cdot\ |$ be a smooth metric of the line bundle associated to $D$, we consider the Monge-Ampere equation as 
\begin{equation}\label{equ general conical MA equation}
(\omega_{0}+\sqrt{-1}\partial \bar{\partial}\phi)^{n}=\frac{e^{h}}{|S|^{2-2\beta}} \omega_{0}^{n},\ h\in C^{1,1,\beta}(M).
\end{equation}
\begin{lem}\label{lem uniqueness of C11 solutions to the MA equation}Suppose both $\phi_{1}$ and $\phi_{2}$ are $C^{1,1,\beta}(M)$ solutions to equation (\ref{equ general conical MA equation}), such that both
$\omega_{0}+\sqrt{-1}\partial\bar{\partial}\phi_{1}$  and 
$\omega_{0}+\sqrt{-1}\partial\bar{\partial}\phi_{2}$ are weak conical K\"ahler-metrics over $M$. Then
$$\phi_{1}-\phi_{2}\equiv\textrm{Constant over}\ M\setminus D.$$
\end{lem}
\begin{proof}{of Lemma \ref{lem uniqueness of C11 solutions to the MA equation}:}  Substracting (\ref{equ general conical MA equation} )
with $\phi=\phi_{2}$ from (\ref{equ general conical MA equation} )
with $\phi=\phi_{1}$, we end up with
\begin{equation}\label{equ phi1-phi2 is harmonic}
\Delta_{s}(\phi_{1}-\phi_{2})=0\ \textrm{over}\  M\setminus D,  
\end{equation}
where $$\Delta_{s}=\int_{0}^{1}g^{i\bar{j}}_{s\phi_1+(1-s)\phi_2}
ds\frac{\partial^{2}}{\partial z_i\partial\bar{z}_j }.$$
Thus, Lemma \ref{lem uniqueness of C11 solutions to the MA equation} follows from equation (\ref{equ phi1-phi2 is harmonic}) and the strong maximal-principle in Theorem 8.1 in \cite{ChenWanglongtime}.
\end{proof}
\begin{thm}\label{thm regularity of weak conical over closed manifold}Suppose $h\in C^{1,1,\beta}(M)$. Then any weak solution to 
(\ref{equ general conical MA equation}) is strong i.e in $C^{2,\alpha,\beta}$, for any $0<\alpha<\min\{\frac{1}{\beta}-1,1\}$.

 In particular, any weak-conical K\"ahler-Einstein metric of  $[M,(1-\beta)D]$ ($0<\beta<1$) must be a $C^{\alpha,\beta}$ conical K\"ahler-Einstein metric, for any $ 0< \alpha < \min({1\over \beta}-1, 1).\;$
\end{thm}
\begin{proof}

Define 
$$K(\phi)=\int_{M}\log \frac{|S|^{2-2\beta}\omega_{\phi}^{n}}{e^{h}\omega_{0}^n}\frac{\omega_{\phi}^{n}}{n!}.$$
Then, along the corresponding conical K\"ahler-Ricci flow,  
 \begin{equation}\label{equ general conical flow}
(\omega_{0}+\sqrt{-1}\partial \bar{\partial}\phi)^{n}=\frac{e^{h+\frac{\partial \phi}{\partial t}}}{|S|^{2-2\beta}} \omega_{0}^{n},
\end{equation}
we deduce
$$K(\phi(t))=\int_{M}\frac{\partial \phi}{\partial t}\frac{\omega_{\phi}^{n}}{n!}.$$
Then, along the flow (\ref{equ general conical flow}), we obtain the following monotonicity by direct computation.
\begin{equation}\label{equ monntonicity of general K energy}
\frac{d K(\phi(t))}{dt}=-\int_{M}|\nabla_{\phi}\frac{\partial \phi}{\partial t}|^{2}\frac{\omega_{\phi}^{n}}{n!}.
\end{equation}

Then, applying  the proof of Theorem 1.7 in \cite{ChenWanglongtime}, with modifications in the $C^{2}$-estimate part (which we will specify later), together with  the monotonicity of the  K-energy  (\ref{equ monntonicity of general K energy}) in the convergence argument 
in Section 11 of \cite{ChenWanglongtime}, we deduce that the flow (\ref{equ general conical flow}) converges to a $\phi_{\infty}\in C^{2,\alpha,\beta}(M)$ which solves equation (\ref{equ general conical MA equation}). The point is, the solution $\phi_{\infty}$ produced by the 
conical K\"ahler-Ricci flow in \cite{ChenWanglongtime} is in $C^{2,\alpha,\beta}(M)$ (strong conical)!. 

Then, both  $\underline{\phi}$ and $\phi_{\infty}$ solve equation (\ref{equ general conical MA equation}). By the uniqueness of $C^{1,1,\beta}$ solutions in Lemma \ref{lem uniqueness of C11 solutions to the MA equation}, we obtain 
\[\underline{\phi}=\phi_{\infty}+Const\in C^{2,\alpha,\beta}.\]
The proof of Theorem \ref{thm regularity of weak conical over closed manifold} is complete. 

   The modification on the $C^{2}$-estimate is that, in the setting of Theorem \ref{thm regularity of weak conical over closed manifold}, it's super easy to apply the Guenancia-Paun type $C^{2}$-estimate as in \cite{GP}, while we surely believe the Chern-Lu inequality as in 
   \cite{CDS2} and \cite{JMR}, and the trick in \cite{Yao} all work equally well. Namely, using the assumption that $\Delta_{\beta}h\geq -C$, formula (22) in  \cite{WYQWF} (for $\epsilon=0$) says
   
   \begin{eqnarray}\label{eqn Siu Bochner consequence}
 & &(\Delta_{\phi}-\frac{\partial }{\partial t})\{\log {tr}_{\omega_{D}}\omega_{\phi}+B|S|^{2\beta}-A\phi\}\nonumber
\\&\geq & {tr}_{\omega_{\phi}}\omega_{D}+A \frac{\partial \phi}{\partial t}  -C.
\end{eqnarray}
By using the  barrier function in the proof of  Theorem 6.2 in \cite{ChenWanglongtime},  the rest of the proof of the $C^{2}$-estimate goes exactly as
   the proof of Lemma 3.1 in \cite{WYQWF}, with $\epsilon=0$.
\end{proof}

Xiuxiong Chen, Department of Mathematics, Stony Brook University,
NY, USA;\ \ xiu@math.sunysb.edu.

Yuanqi Wang, Department of Mathematics, University of California  at Santa Barbara, Santa Barbara,
CA,  USA;\ \ wangyuanqi@math.ucsb.edu.

  \end{document}